\documentclass[11pt]{amsart}
\pdfoutput=1

\calclayout

\usepackage{amsmath, amssymb, amsthm, amsbsy, bbm, amsfonts, enumitem, color, verbatim, array, floatflt, mathtools, braket}
\usepackage{tikz-cd}
\usepackage[all]{xy}
\usepackage[english]{babel}
\usepackage[utf8x]{inputenc}
\usepackage{graphicx, xcolor}

\DeclareGraphicsExtensions{.pdf, .jpg}
\graphicspath{{pictures/}}

\usepackage[pdfpagelabels]{hyperref}

\theoremstyle{plain}
\newtheorem{thm}{Theorem}[section]
\newtheorem*{mainthm}{Main Theorem}
\newtheorem*{thm*}{Theorem}
\newtheorem{prop}[thm]{Proposition}
\newtheorem{lemma}[thm]{Lemma}
\newtheorem{cor}[thm]{Corollary}

\newtheorem{observation}[thm]{Observation}

\theoremstyle{definition}
\newtheorem{definition}[thm]{Definition}
\newtheorem{example}[thm]{Example}

\theoremstyle{remark}
\newtheorem{remark}[thm]{Remark}

\numberwithin{equation}{thm}

\def \L {\mathcal{L}}
\def \A {\mathcal{A}}
\def \C {\mathcal{C}}
\def \E {\mathcal{E}}
\def \B {\mathcal{B}}
\def \R {\mathbb{R}}
\def \F {\mathbb{F}}
\def \Z {\mathbb{Z}}
\def \N {\mathbb{N}}
\def \gl {\mathrm{GL}}

\def \ncp {\mathrm{NCP}}
\def \ncb {\mathrm{NCB}}
\def \nc {\mathrm{NC}}
\def \ncw {\mathrm{NC}(W)}
\def \ocncw {\vert\mathrm{NC}(W)\vert}

\def \red {\mathcal{R}_T}
\def \chains {\mathcal{C}} 
\def \chambers {\mathcal{C}h} 
\def \hw {H(W)} 

\def \dka {(\!(} 
\def \dkz {)\!)} 

\DeclareMathOperator \ecc {ecc} 
\DeclareMathOperator \rad {rad} 
\DeclareMathOperator \diam {diam} 
\DeclareMathOperator \sd {sd} 
\DeclareMathOperator \id {id} 
\DeclareMathOperator \im {im} 
\DeclareMathOperator \rk {rk} 
\DeclareMathOperator \fix {Fix} 
\DeclareMathOperator \mov {Mov} 

\title{Generalized non-crossing Partitions and Buildings}
\author{Julia Heller and Petra Schwer}
\address{Julia Heller and Petra Schwer, Department of Mathematics, Karlsruhe Institute of Technology, Englerstrasse 2, 76131 Karlsruhe, Germany}
\email{julia.heller@kit.edu, petra.schwer@kit.edu}
\date{\today}

\begin{document}
	
	\maketitle
	
	
	\begin{abstract}
		
		For any finite Coxeter group $W$ of rank $n$ we show that the order complex of the lattice of non-crossing partitions $\ncw$ embeds as a chamber subcomplex into a spherical  building  of type $A_{n-1}$. 
		We use this to give a new proof of the fact that the non-crossing partition lattice in type $A_n$ is supersolvable for all $n$. Moreover, we show that in case $B_n$, this is only the case if $n<4$.
		We also obtain a lower bound on the radius of the Hurwitz graph $H(W)$ in all types and re-prove that in type $A_n$ the radius is $\binom{n}{2}$.     
		
	\end{abstract}
	
	\section{Introduction}
	
	The lattice of \emph{non-crossing partitions} $\nc(W,c)$ of a finite Coxeter group $W$, defined with respect to some Coxeter element $c$, is the interval below $c$ in the absolute order, that is
	\[
	\nc(W,c)  = \{\pi \in W: \ \pi \leq c\}.
	\]
	It is easy to see that the isomorphism type of this lattice is independent of the choice of $c$. In type $A$, the case of the symmetric group $S_n$, this definition agrees with  the classical notion of non-crossing partitions $\ncp_n$ introduced by Kreweras \cite{kre}.
	Our main result is the following.
	
	\begin{mainthm}
		For every finite Coxeter group $W$ of rank $n$, the order complex of the non-crossing partitions $\ocncw$ is isomorphic to a chamber subcomplex of a spherical building $\Delta$ of type $A_{n-1}$. This subcomplex  is the union of a collection of apartments and has the homotopy type of a wedge of spheres. 
		Moreover, if $W$ is crystallographic then one can choose $\Delta$ to be finite.  
	\end{mainthm}
	
	The fact that the order complex of $\ncp_n$ embeds into a spherical building of type $A_{n-2}$ was shown by  Brady and McCammond \cite{bra-mcc} as well as Haettel, Kielak and the second author \cite{hks}. Our main result generalizes these results. 
	
	Throughout the rest of the paper we provide a couple of applications of our main result. Corollary~\ref{cor:Asupersolvable} contains a new proof for the supersolvability of the type $A$ non-crossing partitions. This fact was first shown by Hersh as  Theorem 4.3.2 of her Ph.D. Thesis \cite{her}. In contrast to that we show in Theorem~\ref{thm:ncbn_ss}  that $\nc(B_n)$ is only supersolvable if $n \leq 3$.

	In Theorem~\ref{thm:radius} we show that for all finite Coxeter groups $W$ of rank $n$ the radius of the Hurwitz graph, as defined in  Definition~\ref{def:hurwitz}, is bounded below by $\binom{n}{2}$. In addition we provide a simple new proof of the fact that the radius of the Hurwitz graph in type $A_n$ equals $\binom{n}{2}$ in the same theorem. 
	Moreover, we provide examples showing that sometimes $\diam(H(W))>\rad(H(W))$. 
	This partially answers a question formulated by Adin and Roichman in \cite{ar} and disproves their conjecture on the radius of the Hurwitz graph in type $B_3$. 
	
	We would like to illustrate our main result with a concrete example. Figure~\ref{fig:labeling} below shows the order complex of non-crossing partitions $\ncp_4$ inside the spherical building of rank $2$ over $\F_2$. The light gray edges are the chambers of the building that are not contained in the image of $\vert \ncp_4 \vert $. See also Example~\ref{ex:first}. 
	A type $B$ example is given in Figure~\ref{fig:b3} on page \pageref{fig:b3}.

	\begin{floatingfigure}[hp]{0.55\textwidth}
		\centering
		\includegraphics[width=6.5cm]{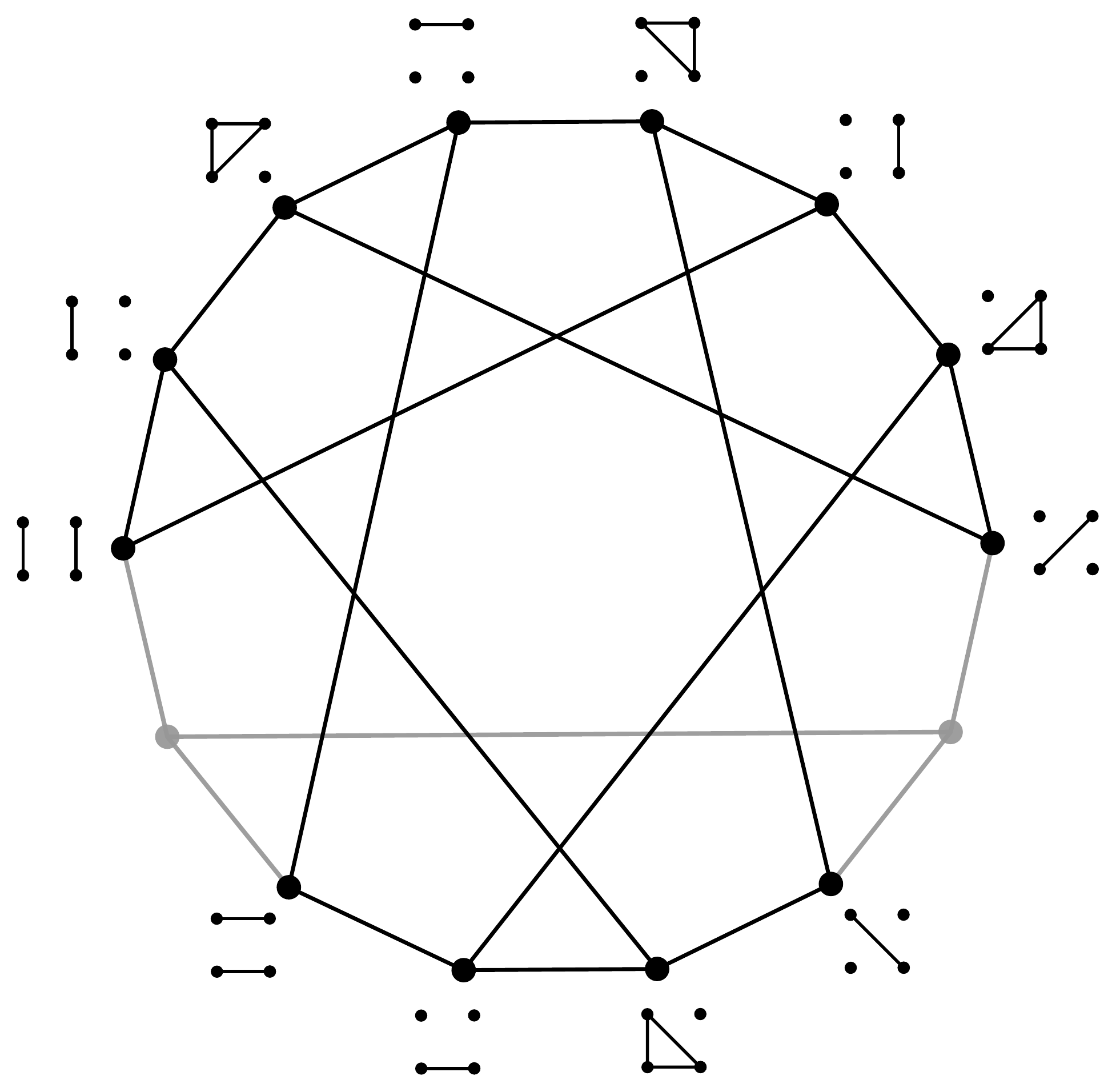}%
		\caption{$\vert\ncp_4\vert$ inside a spherical building.} 
		\label{fig:labeling}
		\label{fig:ncp4}
	\end{floatingfigure}
	
	In Section \ref{sec:ncp} we collect some basic concepts and results concerning absolute order and non-crossing partitions. 
	We also characterize maximal Boolean lattices in the non-crossing partition lattice. 
	
	Structural results of the complex $|\ncw|$ are proven in Section~\ref{sec:embedding}. For instance, we show that the chambers and a subset of the apartments in $|\ncw|$ can be described by reduced factorizations of the Coxeter element. This enables us to show that $|\ncw|$ is a union of apartments. In this section we also prove the main result.
	
	Section \ref{sec:AB} is devoted to explicitly constructing embeddings of the order complex of non-crossing partitions for type $A$ and $B$ into finite buildings. We also provide pictorial interpretations of these embeddings, which we use to give a pictorial description of apartments in type $B$. The analog for type $A$ is in \cite{hks}. This section is also where we discuss supersolvability. 
	
	In Section~\ref{sec:Hurwitz} we will prove the mentioned estimates and equalities on the diameter and radius of the Hurwitz graph $\hw$. This result relies on the embedding of the Hurwitz graph into the chamber graph $\Gamma_\Delta$ of the building $\Delta$ which one obtains from the Main Theorem.  Finally, we discuss a pictorial interpretation of the Hurwitz action in type $A$ in terms of labeled trees. Using this we sketch how to re-prove the transitivity of the Hurwitz action in type $A$. 
	
	In  Appendix \ref{cha:buildings} some very basic facts about spherical buildings of type $A$ are given. The reader not familiar with buildings may want to read that section first.

	\section{Finite Coxeter Groups and non-crossing Partitions}\label{sec:ncp}
	
	Throughout, we will assume that $(W, S)$ is a finite Coxeter system, i.e. $W$ is a finite Coxeter group and $S=\{s_1, \ldots, s_n\}$ the standard generating set for $W$, consisting of elements of order $2$. The \emph{rank} of $W$ is $n$. 
	
	A \emph{Coxeter element} in $W$ is any element conjugate to $s_1\ldots s_n$. Note that in particular every element of the form $s_{\sigma(1)}s_{\sigma(2)}\ldots s_{\sigma(n)}$ for a permutation $\sigma$ in the symmetric group $S_n$ is a Coxeter element. There are slightly different notions of Coxeter elements in the literature. We follow \cite{armstr}, but Humphreys \cite{hum} only considers elements of the form $s_{\sigma(1)}s_{\sigma(2)}\ldots s_{\sigma(n)}$ as Coxeter elements.
	
	In \cite{bes}, Bessis introduced the notion of a \emph{dual Coxeter system} $(W,T)$. It arises from the classical Coxeter system by replacing the generating set $S$ by the larger generating set
	\[
	T = \{wsw^{-1}:\ s \in S, w \in W\}.
	\]
	Since $T$ is the conjugacy closure of $S$ in $W$, it obviously generates $W$ and consists of elements of order $2$. The elements of $T$ are called \emph{reflections} and elements of $S$ are called \emph{simple reflections}. For a generating set $X$ of $W$, the set of reduced expressions with respect to $X$ of $w \in W$ is denoted by $\mathcal{R}_X(w)$. The elements of $\mathcal{R}_X(w)$ are called \emph{$X$-reduced expressions} of $w$.

	\subsection{Absolute order and non-crossing partitions}
	
	We will always assume our generating set of $W$ to be $T$. For example, we will call a word in $W$ \emph{reduced}, if it is reduced with respect to $T$ and we denote by $\ell$ the word length on $W$ with respect to $T$. Define the \emph{absolute order} $\leq$ on $W$ by setting
	\[
	v \leq w \ \iff \ \ell(w) = \ell(v) + \ell(v^{-1}w)
	\]
	for all $v,w \in W$.  The poset $(W, \leq)$ is a finite poset graded by $\ell$ with the identity as unique minimal element and the Coxeter elements among the maximal elements. Note that $v \leq w$ holds if and only if there is a reduced expression of $w$ such that a prefix of it is a reduced expression for $v$.
	The absolute order is the natural analog of the weak order in the classical approach. 
	
	Let us now review two properties of reduced words and the absolute order.
	
	\begin{lemma}[Shifting property]\label{lem:shift}\cite[Lem. 2.5.1]{armstr}.
		For $w \in W$ and $t_1\ldots t_k \in \red(w)$, the maps $\sigma_i, \sigma_i'\colon \red(w) \to \red(w)$ defined by
		\begin{align*}
		\sigma_i(t_1\ldots t_k) &= t_1\cdot t_2\cdot \ldots \cdot t_{i-1}\cdot (t_i t_{i+1} t_i)\cdot t_i \cdot t_{i+2} \cdot\ldots\cdot t_k,\\
		\sigma'_i(t_1\ldots t_k) &=t_1 \cdot t_2 \cdot \ldots \cdot t_{i-1} \cdot t_{i+1} \cdot (t_{i+1} t_i t_{i+1}) \cdot t_{i+2} \cdot\ldots\cdot t_k
		\end{align*}
		are inverse self-maps of $\red(w)$ for all $1\leq i <k$. In particular, $T$-reduced expressions do not contain repetitions of letters of $T$. 
	\end{lemma}
	
	Note that these shifts only alter the $i$th and $i+1$st letter of a reduced decomposition. 
	It is a straightforward computation that $\sigma_i'=\sigma_i^{-1}$ and that the $\sigma_i$ satisfy the braid relations. They hence induce an action of the braid group on $\red(c)$, which is called the \emph{Hurwitz action}. This action is known to be transitive \cite[Prop 1.6.1]{bes}. 
	
	These shifts are used to prove the following characterization of the absolute order, which is an analog of the subword property of the Bruhat order on $W$. If $t_1t_2\ldots t_k$ is a word, we call an expression $t_{i_1}t_{i_2}\ldots t_{i_l}$ a \emph{subword} if $1 \leq i_1 < i_2 < \ldots < i_l \leq k$.
	
	\begin{prop}[Subword property]\label{prop:subword}\cite[Prop. 2.5.2]{armstr}.
		For $v,w \in W$ we have $v \leq w$ if and only if there is a $T$-reduced expression of $w$ that contains $v$ as a subword. 
	\end{prop}
	
	\begin{definition}\label{def:ncw}
		Let $c$ be a Coxeter element in $W$. The \emph{non-crossing partitions} $\nc(W,c)$  of $W$ with respect to $c$ is the interval between $\id$ and $c$ in absolute order, i.e.
		\[
		\nc(W,c) \coloneqq [\id, c] = \{\pi \in W: \ \pi \leq c\}.
		\]
	\end{definition}
	
	The fact that any two Coxeter elements $c$ and $c'$ are conjugate implies that $\nc(W,c)$ and $\nc(W,c')$ are isomorphic posets. We therefore write $\ncw$ for $\nc(W,c)$.  We also denote by $\nc(X_n)$ the non-crossing partitions of the Coxeter group of type $X_n$. 
	
	The grading of $W$ by $\ell$ induces a grading of $\ncw$ by $\ell$. Moreover, $\ncw$ carries the structure of a  finite, graded lattice of rank $n-1$, where $n$ is the rank of $W$. 
	See \cite[Thm. 7.8]{bra_watt_lattice} for a uniform proof of this fact. 
	
	The next lemma shows that taking joins is compatible with the group structure.  
	
	\begin{lemma}\label{lem:product_join}\cite[Le. 2.6.13]{armstr}.  
		Choose $w$ in $\ncw$. For any $T$-reduced expression $t_1 \ldots t_k$ of $w$ we have $w = t_1 \vee \ldots \vee t_k$.  
	\end{lemma}
	
	Maximal chains in $\ncw$ are all of the same length and have an interpretation in terms of reduced decompositions. Compare also Corollary 4 of \cite{bra_watt_par_ord}. Let $\C(W)$ denote the set of maximal chains in $\nc(W,c)$.
	
	\begin{lemma}\label{lem:max_chains}
		The map $g\colon \red(c) \to \C(W)$ defined by
		\[
		g(t_1\ldots t_n) = (\id \leq t_1 \leq t_1t_2 \leq \ldots \leq t_1\ldots t_n)
		\]
		is a bijection of reduced decompositions of the Coxeter element $c$ and maximal chains in $\nc(W,c)$.
	\end{lemma}
	
	\begin{proof}
		For a given maximal chain $\id \leq w_1 \leq w_2 \leq \ldots \leq w_{n-1}=c$ a reduced decomposition of $c$ is given by $w_1 \cdot w_1^{-1}w_2 \cdot w_2^{-1}w_3 \cdot \ldots \cdot w_{n-2}^{-1}w_{n-1}$. The fact that $w_{i}^{-1}w_{i+1}$ is a reflection in $T$ follows from the maximality of the chain. This obviously provides an inverse map.
	\end{proof}
	
	A simple computation implies the next corollary. 
	
	\begin{cor}\label{cor:adj_chains}
		Choose $t_1\ldots t_n \in \red(c)$. For all $1\leq i<n$ and all integers $k \geq 1$, the chains $g(\sigma_i^k(t_1\ldots t_n))$ and $g(t_1\ldots t_n)$ differ by exactly one element. The analog holds for $\sigma'$.
	\end{cor}

	\subsection{A description of $\ncw$ in terms of moved spaces}\label{subsec:movedspaces}
	
	Let $W$ be a finite Coxeter group and $c$ a fixed Coxeter element of $W$. 
	We denote the standard geometric representation of $W$ on $V=\R^S$ by $\rho$ and the root corresponding to a reflection $t$ in $W$ by $\alpha_t$ and conversely, the reflection corresponding to a root $\alpha$ by $t_\alpha$. 
	
	For each $w \in W$, the \emph{fixed space}  $\fix(w)$ is the subspace of $V$ that is pointwise  fixed  by the linear transformation $\rho(w)$, i.e. $\fix(w)=\ker(\rho(w)-\id)$. Analogously, the \emph{moved space} $\mov(w)$ of $w$ is the orthogonal complement of $\fix(w)$ in $V$, i.e. $\mov(w)=\im(\rho(w) - \id)$. 
	
	We obtain the following description of the moved space in terms of reduced expressions and roots from the proof of Theorem~2.4.7 of \cite{armstr}.

	\begin{observation}\label{obs:basis}
		If $t_{\alpha_1}t_{\alpha_2}\ldots t_{\alpha_k}$ is a reduced expression for $w \in W$, then the set $\{\alpha_1, \alpha_2, \ldots, \alpha_k\}$ is a basis of $\mov(w)$. 
	\end{observation}
	
	This observation is a refinement of the following \cite[Lem. 3]{car}.
	
	\begin{lemma}[Carter's Lemma]\label{lem:carter}
		For $w \in W$, a word $w = t_{\alpha_1}t_{\alpha_2}\ldots t_{\alpha_k}$ is reduced if and only if the roots $\alpha_1, \alpha_2, \ldots, \alpha_k$ are linearly independent.
	\end{lemma}
	
	The following proposition is central to the construction of an embedding of the order complex of non-crossing partitions into a spherical building. This follows from results of Brady and Watt \cite{bra_watt_par_ord}, \cite{bra_watt_kpi} and is stated in a similar way in Theorem 2.4.9 in \cite{armstr}. For a vector space $V$, we denote by $\L(V)$ the lattice of linear subspaces of $V$.
	
	\begin{prop} \label{prop:emb}
		For every finite Coxeter group $W$ with Coxeter element $c$ and $V = \mov(c)$ we have  $\rk(\ncw) = \rk(\L(V))$ and the map
		\begin{align*}
		f \colon \ncw \to \L(V),\quad
		\pi \mapsto \mov(\pi)
		\end{align*}
		is a poset map that is injective, not surjective and rank preserving.  
	\end{prop}
	
	\begin{proof}
		The map $f$ is well-defined by Section 2 of \cite{bra_watt_kpi} and the injectivity of $f$ is Theorem 1 of \cite{bra_watt_par_ord}. The fact that $f$ is rank-preserving is shown in Proposition 2.2 of \cite{bra_watt_kpi}. In particular, we get that $\rk(\ncw) = \rk(\L(V))$, since $\ell(c) = \dim(V)$ by construction. Since $W$ is finite and $V \cong \R^n$, the image of $f$ is always a proper subset of $\L(V)$. This completes the proof.
	\end{proof}
	
	We will identify $\ncw$ with its image under $f$ and will sometimes abuse notation and write  $\ncw\subseteq \L(V)$. 
	
	The following lemma is needed in Section~\ref{sec:AB} to construct explicit embeddings for type $A$ and $B$.
	
	\begin{lemma}\label{lem:join_assoc}
		Let $w \in \ncw$ and $t \in T$ such that $\rk(w \vee t)=\rk(w)+1$. Then $f(w \vee t) = f(w) \vee f(t)$. 
		In particular, for a $T$-reduced decomposition $t_1\ldots t_k$ of $w$ we get $f(w)=f(t_1 \vee \ldots \vee t_k)=f(t_1) \vee \ldots \vee f(t_k)$.
	\end{lemma}
	
	\begin{proof}
		Since $f$ is order preserving we get that $f(w), f(t) \leq f(w \vee t)$, hence $f(w) \vee f(t) \leq f(w \vee t)$. Moreover, $\rk(f(w) \vee f(t))=\rk f(w)+ \rk f(t) = \rk f(w \vee t)$, where the second equality holds by assumption. Uniqueness of the join then implies $f(w) \vee f(t) = f(w \vee t)$. The second assertion follows with Lemma \ref{lem:product_join}.
	\end{proof}
	
	\subsection{Boolean lattices in $\ncw$}
	We characterize a subset of maximal Boolean sublattices of $\ncw$ in terms of reduced decompositions of the Coxeter element. 

	A lattice is a \emph{Boolean lattice}, if it is isomorphic to the lattice $\B_n$ of subsets of $\{1, \ldots, n\}$ ordered by inclusion for some $n$. The set of all maximal Boolean lattices in $\ncw$ is denoted by $\B(\ncw)$.

	\begin{lemma}\label{lem:aptms}
		The map $h\colon \red(c) \to \B(\ncw)$ defined by 
		\[ 
		t_1\ldots t_n \mapsto \left\{\bigvee_{i\in I}t_i:\ I\subseteq \{1, \ldots, n\}\right\}
		\]
		assigns to every reduced decomposition of the Coxeter element $c$ a maximal Boolean lattice in $\ncw$.
	\end{lemma}

	\begin{proof}
		By the subword property~\ref{prop:subword} and the compatibility of taking products and the join in $\ncw$, which was shown in Lemma~\ref{lem:product_join}, this map is well-defined. 
	\end{proof}

	Note that in general $h$ is neither injective nor surjective. 
	For instance, $(12)(34)(24)$ and $(34)(12)(24)$ are two different reduced expressions for the Coxeter element $(1234)$ in the symmetric group $S_4$, but the corresponding Boolean lattices are clearly the same as the sets of reflections coincide.

	In type $A$ the map is surjective by Lemma 2.5 of \cite{gy}. But the surjectivity fails in other types, as can be seen in type $B$. For the type $B$ specific notation see Section \ref{sec:typeB}. For instance, the Boolean lattice spanned by the reflections $[1], [2], \ldots, [n]$ is \emph{not} contained in the image of $h$, since no product of the above reflections is the Coxeter element $c=[12\ldots n]$ of type $B_n$.

	\section{Non-crossing Partitions embed into spherical Buildings}\label{sec:embedding}
	
	In this section we prove our Main Theorem as stated in the introduction and show that $|\ncw|$ shares important properties with buildings, although (in general) it is not a building itself.
	We split the proof of the Main Theorem into several smaller statements and start by showing that there is an embedding into a building. 
	
	\begin{prop}\label{prop:oc-emb}
		The order complex $\ocncw$ of the non-crossing partition lattice associated to a finite Coxeter group $W$ of rank $n$ embeds into a spherical building $\Delta$ of type $A_{n-1}$. In particular, $\dim(\ocncw)=\dim(\Delta)$.
	\end{prop}
	
	\begin{proof}
		The map $f$ of Proposition~\ref{prop:emb} induces a map $\bar f:\ocncw\to \vert \L(V)\vert$ on the order complexes. By Proposition~\ref{prop:LV_building} the order complex of $\L(V)$ is a spherical building $\Delta$ of type $A_{n-1}$. As $f$ is injective and rank preserving, the induced map $\bar f$ is an embedding of simplicial complexes. The equality of dimensions is implied by the equality of ranks, which is guaranteed by Proposition~\ref{prop:emb}.
	\end{proof}
	
	The following two subsections contain, besides additional results, the proofs of the remaining properties of this embedding stated in the main theorem.  We summarize the proof of the Main Theorem here.  
	
	\begin{proof}[Proof of Main Theorem]
		By Proposition \ref{prop:oc-emb}, the complex $\ocncw$ embeds into the building $\Delta=|\L(V)|$. From the same proposition it follows that $\dim(\ocncw)=\dim(\Delta)$ and since $\ocncw$ is a chamber complex by Lemma \ref{lem:cc}, it is isomorphic to a chamber subcomplex of $\Delta$.
		
		The fact that the image of $\ocncw$ in the building is a union of apartments is shown in Corollary \ref{cor:ncw_union_aptm}.
		
		For the assertion on the homotopy type see Remark \ref{rem:wedge}.
		In type $A$, this also follows directly from Proposition \ref{prop:ncpn_union}.
		
		Finally, the fact about crystallographic Coxeter groups is shown in Proposition~\ref{prop:crys}. 
	\end{proof}

	\begin{example}\label{ex:first}
		Figure~\ref{fig:labeling} on page~\pageref{fig:labeling} shows the image of the order complex of the non-crossing partitions  $\nc(S_4)$ sitting inside the spherical building of rank 2 over $\F_2$.    
		Each vertex of $|\ncp_4|$ is labeled by the non-crossing partition that is its preimage. For the pictorial description of non-crossing partitions see Section~\ref{Sec:typeA}. The higher rank simplices in the image of $|\ncp_4|$ in the building are uniquely  determined by its vertices (as $|\ncp_4|$ is a flag complex). They correspond to chains of non-crossing partitions whose length is the dimension of the simplex. 
	\end{example}

	\subsection{Building-like structure of $\ocncw$}\label{subsec:bls}
	
	The \emph{chambers} are the maximal simplices in $\ocncw$ and we write $\chambers(W)$ for the set of all chambers in $\ocncw$. In Definition~\ref{def:NCPapartments} we will introduce the notion of apartments in $\ocncw$ and relate chambers and apartments to reduced factorizations of the Coxeter element. We prove in Lemma~\ref{lem:cc} that $\ocncw$ is a chamber complex, which can be written as a union of apartments, see Corollary \ref{cor:ncw_union_aptm}. 
	
	\begin{remark}\label{rem:bij_chains-chambers}
		By definition of the order complex, the maximal chains in $\ncw$ are in bijection with the maximal simplices  in $\ocncw$, i.e. its chambers,  via
		\[
		\tilde{g}\colon (\id \leq t_1 \leq t_1t_2 \leq \ldots \leq t_1\ldots t_n) \mapsto \{t_1, t_1t_2, \ldots, t_1\ldots t_{n-1}\}.
		\]
	\end{remark}

	\begin{lemma}\label{lem:chambers}
		The map $g'\colon \red(c) \to \chambers(W)$ defined by
		\[
		g'(t_1\ldots t_n) = \{t_1, t_1t_2, \ldots, t_1\ldots t_{n-1}\}
		\]
		is a bijection of the set of reduced decompositions of the Coxeter element $c$ and the set of chambers in $\ocncw$.
	\end{lemma}
	
	\begin{proof}
		The map $g'$ is the composition of the bijections $g$ from Lemma \ref{lem:max_chains} and $\tilde{g}$ from above.
	\end{proof}
	
	The shifts $\sigma_i$ on the reduced decompositions of the Coxeter element induce maps on $\chambers(W)$ as follows. For any chamber $C\in\chambers(W)$ put $\sigma(C) \coloneqq g'(\sigma(g'^{-1}(C)))$ and define $\sigma_i'$ analogously.
	The chamber $\sigma(C)$ corresponds to the reduced expression that arises from the reduced decomposition defining $C$ by applying the shift $\sigma$. From this we obtain the following corollary.
	
	\begin{cor}\label{cor:adj_chambers}
		For all $C \in \chambers(W)$, $1\leq i < n$ and $k \geq 1$, the chamber $\sigma_i^k(C)$  is adjacent to $C$, i.e. they have all but one vertex in common. The same holds for $\sigma_i'$.
	\end{cor}
	
	\begin{proof}
		The chains in $\ncw$ corresponding to $C$ and $\sigma_i^k(C)$ differ in exactly one element by Corollary \ref{cor:adj_chains}. Hence, the corresponding chambers differ by one element as $\tilde{g}$ maps a chain to the simplex on its elements. Similarly obtain the statement for $\sigma_i'$.
	\end{proof} 
	
	Figure~\ref{fig:adj_chambers} shows an example of chambers sharing a codimension one face and the corresponding reduced decomposition. Here $n=3$ and $i=2$.
	
	\begin{figure}[h]
		\begin{center}
			\def\svgwidth{6cm}
			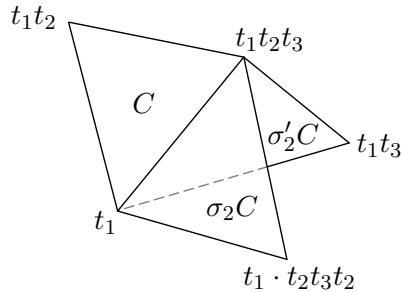
			\caption{Three adjacent chambers $C, \sigma_2(C)$ and $\sigma'_2(C)$ given by the reduced decompositions $t_1t_2t_3$, $t_1(t_2t_3t_2)t_2$ and $t_1t_3(t_3t_2t_3)$, respectively.}
			\label{fig:adj_chambers}
		\end{center}
	\end{figure}
	
	The following is an easy consequence of the fact that the non-crossing partition lattices are shellable \cite{abw}. It also can be shown directly using the transitivity of the Hurwitz action.
	
	\begin{lemma}\label{lem:cc}
		The complex $\ocncw$ is a chamber complex.
	\end{lemma}

	\begin{definition}\label{def:NCPapartments}
		An \emph{apartment} in $\ocncw$ is a subcomplex $A$ such that $A = |B|$ for a Boolean lattice $B \in \B(\ncw)$. The set of apartments in $\ocncw$ is denoted by $\A(W)$.  
	\end{definition}
	
	Note that every apartment is isomorphic (as abstract simplicial complex) to the Coxeter complex of type $A_{n-1}$. Therefore $\A(W)$ is a subset of the set of apartments in the building $\Delta$.    
	Hence, we get a bijection
	\[
	\tilde{h}\colon \B(\ncw) \to \A(W), \quad B \to |B|.
	\]

	There are apartments in $\ocncw$ which can be described in terms of reduced decompositions of the Coxeter element.

	\begin{lemma}\label{lem:red_dec_aptm}
		The map $h'\colon \red(c) \to \A(W)$ defined by
		\[
		t_1\ldots t_n \mapsto \left|\left\{\bigvee_{i\in I}t_i:\ I\subseteq \{1, \ldots, n\}\right\}\right|
		\]
	
	assigns to every reduced decomposition of the Coxeter element $c$ an apartment contained  in $\ocncw$.
	\end{lemma}
	
	In the case $W$ is of type $A_n$, this is Lemma $2.5$ of \cite{gy}.
	
	\begin{proof}
		Composition of the map $h$ from Lemma \ref{lem:aptms} and the bijection $\tilde{h}$ from above gives a map from the reduced decompositions of $c$ into $\A(W)$. 
	\end{proof}

	Hence every apartment of $\ocncw$ which is in the image of $h'$
 is uniquely determined by a set of $n$ reflections $\{t_1, \ldots, t_n\}$ such that there is a permutation $\tau\in S_n$ of indices such that $t_{\tau(1)}\ldots t_{\tau(n)}=c$. Such a set of reflections corresponds to a basis $\{\alpha_{t_1}, \ldots, \alpha_{t_n}\}$ of $V$.
	
	\begin{lemma}\label{lem:ch_in_aptm}
		Let $A \in \A(W)$ be the apartment corresponding to a reduced decomposition $t_1\ldots t_n$ of $c$.
		 The set of chambers of $A$ is given by
		\[
		\chambers(A)=\left\{\left|\{t_{\tau(1)}, t_{\tau(1)}\vee t_{\tau(2)}, \ldots,  t_{\tau(1)}\vee\ldots\vee t_{\tau(n)}\}\right| :\ \tau \in S_n\right\}.
		\]
	\end{lemma}
	
	\begin{proof}
		The maximal chains in a Boolean lattice are given by orderings (i.e. permutations) of the rank one elements. As the elements of rank one in $A$ (as Boolean lattice) are exactly $t_1, \ldots, t_n$, the assertion follows. 
	\end{proof}

	\begin{lemma}\label{lem:ch_aptm}
		Let $t_1\ldots t_n$ be in $\red(c)$. The chamber $g'(t_1 \ldots t_n) \in \chambers(W)$ is contained in the apartment $h'(t_1\ldots t_n) \in \A(W)$.
	\end{lemma}
	
	\begin{proof}
		The chamber $g'(t_1 \ldots t_n)$ is given by 
		\[
		\{t_1, t_1 t_2, \ldots, t_1\ldots t_n\}=\{t_1, t_1\vee t_2, \ldots, t_1\vee\ldots\vee t_n\},
		\]
		which is clearly contained in the apartment $h'(t_1 \ldots t_n)$ by Lemma \ref{lem:ch_in_aptm}.
	\end{proof}
	
	As an easy consequence, we get the following.
	
	\begin{cor}\label{cor:ncw_union_aptm}
		For every finite Coxeter group $W$, the complex $\ocncw$ is a union of apartments.
	\end{cor}
	
	\begin{proof}
		Since $\ocncw$ is a chamber complex, it is the union of its chambers. Since every chamber is contained in an apartment of $\ocncw$ by Lemma \ref{lem:ch_aptm}, it is a union of apartments.
	\end{proof}
	
	\begin{remark}\label{rem:wedge}
		It is an immediate consequence of Corollary \ref{cor:ncw_union_aptm} that $|\ncw|$ is homotopy equivalent to a wedge of spheres: $|\ncw|$ is an $(n-3)$-connected complex of dimension $n-2$, which is by a standard argument in topology, see for example \cite[Thm. 8.6.2]{tom}, homotopy equivalent to a wedge of spheres of dimension $n-2$. The homotopy type of $|\ncw|$ is also a direct consequence of the shellability of $\ncw$ \cite{abw}.
	\end{remark}
	
	From the embedding of the order complex $\ocncw$ into a spherical building $\Delta$, we immediately obtain a retraction onto any apartment it contains. 
	A similar projection (which was explicitly constructed) was used by Adin and Roichman \cite{ar} when they computed the radius of the Hurwitz graph in type $A$. 
	
	\begin{prop}
		For every chamber $C$ and every apartment $A$ in $\ocncw$ containing $C$ we obtain a simplicial map $\rho_{A,C}\colon\ocncw \to A$  that is distance non-increasing and the restriction of $\rho_{A,C}$ to an apartment $A'$ containing $C$ is an isomorphism onto $A$.  
	\end{prop}
	
	\begin{proof}
		It is a well known fact that there is a retraction 
		$\rho_{A,C}\colon\Delta \to A$ of the building for any pair $C$ and $A$ that has the desired properties \cite[Sec. 4.4]{ab}. Its restriction to the image of $\ocncw$ in $\Delta$ is still distance diminishing as 
		\[d_A(\rho_{A,C}(X), \rho_{A,C}(Y))\leq d_\Delta(X,Y)\leq d_{\ocncw}(X,Y)\]
		for all chambers $X$ and $Y$. Here $d_A$, $d_\Delta$, $d_{\ocncw}$ are the length metrics on the chambers of $A$, $\Delta$ and $\ocncw$, respectively, i.e. the lengths of minimal galleries between chambers. Hence the claim. 
	\end{proof}
	
	Note that the proposition immediately implies that for every retraction $\rho_{A,C}$ and every partial order on the chambers of $A$ one obtains a partial order on the chambers of $\ocncw$ by defining $D\leq  E \Leftrightarrow \rho_{A,C}(D)\leq \rho_{A,C}(E)$.

	\subsection{Crystallographic Coxeter groups}\label{cha:cry}
	
	Humphreys characterizes the finite crystallographic Coxeter groups in Proposition 6.6. of~\cite{hum} as being the 
	the finite Coxeter groups $W$ with Coxeter system $(W,S)$ where the edge labels in the Coxeter diagram are in $ \{2,3,4,6\}$ for all $s, t \in S$. 
	In this case, there exists a root system $\Phi$ associated to $W$ such that every positive root can be written as a non-negative integer linear combination of simple roots.
	The set $B$ of simple roots forms a basis for $V$ and one can show that then every transformation matrix $\rho(w)$, $w \in W$ is integral with respect to $B$. 
	See \cite[Sec. 2.9]{hum} or \cite[Sec. 2.2.1]{armstr} for details.

	From now on suppose that $W$ is finite crystallographic and choose a basis $B$ of $V$ as above. 
	Since the finite image $\im(W)$ is contained in $\gl(\Z)$, we find a prime $p \in \Z$ such that the representation $\rho_p\colon W \to \gl_n(\F_p)$ induced by $\Z \to \Z/p\Z$ is faithful. In particular, we get an injective map 
	\[
	\{\alpha_t: t \in T\}\to \F_p^n, \quad \alpha_t \mapsto \beta_t.
	\]
	
	Let $w\in W$ and $t_1\ldots t_k$ be a reduced decomposition for it.
	Recall that a basis for $\mov(w) \subseteq V$ is given by $\{\alpha_{t_1}, \ldots, \alpha_{t_k}\}$, see Observation \ref{obs:basis}.  
	Let $V_p=\bigoplus_{s \in S} \beta_s \F_p$. In analogy to the real case, we define the \emph{moved space} of $w \in W$ in $V_p$ to be the $V_p$-span $\mov_p(w)=\langle\beta_{t_1}, \ldots, \beta_{t_k}\rangle_{V_p}$. 
	
	\begin{prop}\label{prop:crys}
		For every finite crystallographic Coxeter group $W$ of rank $n$, the complex $\ocncw$ embeds into a \emph{finite} spherical building of type $A_{n-1}$ over $\F_p$.
	\end{prop}
	
	\begin{proof}
		Let $f\colon \ncw \to \L(V)$ with $f(\pi)=\mov(\pi)$ be the map from Proposition \ref{prop:emb}. Since $W$ is crystallographic, every $f(\pi)$ has a basis consisting of integer roots. For $V_\Z=\bigoplus_{s\in S}\alpha_s\Z$, we hence get an injective, rank preserving poset map $f'\colon \ncw \to \L(V_\Z)$. Composition with the map $\L(V_\Z) \to \L(V_p)$, induced by $\alpha_s \mapsto \beta_s$, which is injective by assumption, gives the desired embedding of posets $\ncw \to \L(V_p)$. The induced map on order complexes provides the embedding of $|\ncw|$ into the finite building $|\L(V_p)|$.
	\end{proof}
	
	When $W$ is fixed one can explicitly determine a minimal prime $p$ such that $\ocncw$ embeds into $|\L(\F_p^n)|$. 
	For type $A_n$, we get that $p=2$ and for $B_n$ one chooses $p=3$. Compare Section \ref{sec:AB} for details.

	\section{Types $A$ and $B$}\label{sec:AB}
	
	In this section, we construct explicit embeddings of the non-crossing partitions of type $A$ and $B$ in a spherical building over $\F_2$ and $\F_3$, respectively.
	We will interpret the embeddings using the pictorial presentations of $\nc(A_n)$ and $\nc(B_n)$.
	In type $A$, our construction is a special case of the embeddings in \cite{hks}. 
	
	\subsection{Type $A$}\label{Sec:typeA}
	
	The Coxeter group of type $A_{n-1}$ is the symmetric group $S_n$. 
	The set of reflections is 
	\[
	T(S_n)=\{(i, j) \in S_n:\ 1 \leq i < j \leq n\}, 
	\]
	the transpositions of $S_n$.
	From now on, we will fix the Coxeter element $c=(1,2\ldots, n)$ of $S_n$.
	
	We start with explicitly constructing the embedding from Proposition \ref{prop:crys} of $|\nc(S_n)|$ into a finite spherical building. Recall that a root system of type $A_{n-1}$ in $\R^n$ is given by
	\[
	\{\varepsilon_i - \varepsilon_j\ |\ 1 \leq i, j \leq n, i \neq j\},
	\]
	where $\varepsilon_i\in\R^{n}$ is the $i$th standard basis vector. In particular, the roots are contained in the subspace of dimension $n-1$, where all coordinates sum up to zero. Note that the root $\varepsilon_i - \varepsilon_j$ corresponds to the transposition $(i,j) \in T(S_n)$ under the standard geometric representation.

	Let $e_1, \ldots, e_{n-1}$ be the standard basis vectors of $\F_2^{n-1}$.
	
	\begin{prop}\label{prop:An_emb}
		The map $f'\colon T(S_{n}) \to \L(\F_2^{n-1})$ given by 
		\[(i, j) \mapsto 
		\begin{cases}
		\langle e_i + e_{j} \rangle &\text{if } j < n\\
		\langle e_i\rangle &\text{if } j=n
		\end{cases}
		\]
		is injective and induces an embedding $ f\colon |\nc(S_n)| \to |\L(\F_2^{n-1})|$ of the complex of non-crossing partitions of type $A_{n-1}$ into a finite building of type $A_{n-2}$ defined over $\F_2$.
	\end{prop}
	
	\begin{proof}
		The injectivity of the map $f'$ is clear. Since we can extend $f'$ to $\nc(S_n)$ by Lemma \ref{lem:join_assoc} via $f(t \vee t') \coloneqq f'(t) \vee f'(t')$ for a reduced expression $tt'$, we get the desired embedding $f$.
	\end{proof}

	The non-crossing partitions of type $A_{n-1}$ are isomorphic to the classical non-crossing partitions $\ncp_n$ of a cycle, introduced by Kreweras in \cite{kre}. In \cite{bra_kpi} it was shown that $\nc(A_{n-1}) \cong \ncp_n$ as lattices. We will recall the correspondence and give interpretations of the apartments and chambers in $\nc(A_{n-1})=\nc(S_n)$ in the pictorial representation.
	
	Let $\pi= \{B_1, \ldots, B_k\}$ be a partition of $\{1, \ldots, n\}$.
	The pictorial representation $P(\pi)$ of $\pi$ is obtained as follows: label the vertices of a regular $n$-gon in the plane with $1, \ldots, n$ clockwise in this order and draw the convex hull of the elements of $B$ for every $B\in \pi$. The partition is called \emph{non-crossing}, if there are no crossing blocks in $P(\pi)$.  
	When no confusion arises, we will also write $\pi$ for $P(\pi)$. The set of non-crossing partitions of $\{1, \ldots, n\}$ is denoted by $\ncp_n$. It is a graded lattice \cite{kre} with partial order given by refinement. The rank of $\pi \in \ncp_n$ is $\rk(\pi)=n-k$, where is the number of blocks of $\pi$. Join and meet are given by the non-crossing span and intersection, respectively.
	
	An element $w \in \nc(S_n)$ has a unique (up to order) cycle decomposition $w=z_1\ldots z_k$ into disjoint cycles. For a cycle $z=(i_1, \ldots, i_k)$, let $\{z\}$ denote the set $\{i_1, \ldots, i_k\} \subseteq \{1, \ldots, n\}$. For every $w \in \nc(S_n)$ with cycle decomposition $w=z_1\ldots z_k$, the set $\{\{z_1\}, \ldots, \{z_k\}\}$ induces a non-crossing partition $\{w\}$ of $\{1, \ldots, n\}$. Brady proved the following \cite[Le. 3.2]{bra_kpi}.

	\begin{lemma}\label{lem:lattice_iso}
		For all $n \in \N$, the map
		\[
		\nc(S_n) \to \ncp_n, \quad w \mapsto P(\{w\})
		\]
		is a lattice isomorphism. 
	\end{lemma}

	We write $P(w)$ for $P(\{w\})$ and interpret it as embedded graph, hence we may speak of edges of $P(w)$. Each connected component of $P(w)$ (i.e. the convex hulls of blocks) then corresponds to the complete graph on the same vertices (but we only draw the outer edges).

	\begin{remark}
		Combining Proposition \ref{prop:An_emb} and Lemma \ref{lem:lattice_iso} gives an embedding of $|\ncp_n|$ into the building $\L(\F_2^{n-1})$, where the edge $(i,j)$ is mapped to the subspace $\langle e_i + e_j \rangle$ when $j \neq n$ and to the subspace $\langle e_i \rangle$ otherwise. Figure \ref{fig:ncp4} on page \pageref{fig:ncp4} shows the order complex of $\ncp_4$ and how it sits inside a spherical building. The labeling is chosen in such a way that it fits with the labeling of the building $\L(\F_2^3)$ in Figure \ref{fig:F2-building} on page \pageref{fig:F2-building}.
	\end{remark}
	
	The apartments in $|\ncp_n|$ have a graphical description. The following is Proposition $4.4$ of \cite{hks}.
	
	\begin{prop}
		The apartments of $|\ncp_n|$ are in bijection with non-crossing spanning trees on $n$ vertices on a cycle.
	\end{prop}
	
	This correspondence enables us to enumerate the apartments.
	
	\begin{cor}
		There are
		\[
		\frac{1}{2n-1} \binom{3n-3}{n-1}
		\]
		apartments in $|\ncp_n|$.
	\end{cor}
	
	\begin{proof}
		The number of non-crossing spanning trees on $n$ vertices is well-known to be the generalized Catalan number $\frac{1}{2n-1} \binom{3n-3}{n-1}$, see for example \cite[Cor. 1.2]{noy}.
	\end{proof}
	
	\begin{example}\label{ex:tree_chambers}
		By Lemma \ref{lem:red_dec_aptm} every reduced decomposition of the Coxeter element $c=(1,2, \ldots, n)$ gives rise to an apartment in $|\ncp_n|$. The non-crossing tree describing the apartment corresponding to the decomposition $(13)(35)(12)(34)$ of the Coxeter element $c=(12345) \in S_5$ is the non-crossing tree in Figure \ref{fig:chambers}. Note that this is also the apartment corresponding to the decomposition $(13)(12)(35)(34)$. 
	\end{example}
	
	\begin{figure}[h]%
		\begin{center}
			\def\svgwidth{9cm}
			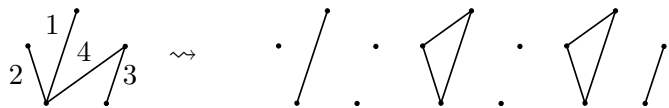
			\caption{A labeling of a non-crossing spanning tree given by $(13)(35)(12)(34)\in \red(c)$ induces a chamber in the apartment corresponding to it.}%
			\label{fig:chambers}%
		\end{center}
	\end{figure}

	\begin{remark}
		The complex $|\ncp_n|$ is \emph{not} a building if $n > 3$. Axiom (B1) of Definition \ref{ax:b2} is not satisfied, as there are simplices, which are not contained in a common apartment. Crossing edges are an example.
	\end{remark}
	
	To finish the discussion about type $A$ non-crossing partitions, we reinterpret the supersolvability of $\ncp_n$ in geometric terms and give a new proof of the fact that $\nc(A_n)$ is supersolvable. Moreover, we characterize the $M$-chains in $\ncw$. The supersolvability plays a crucial role in computing the radius of the Hurwitz graph $H(S_n)$ in Theorem \ref{thm:radius}.

	\begin{definition}
		A finite lattice $L$ is called \emph{supersolvable}, if there exists a maximal chain $d \subseteq L$ such that for every chain $c \subseteq L$ the sublattice generated by $c$ and $d$ is distributive. Such a chain $d$ is called an $M$-chain.
	\end{definition}

	Before we can state and prove the main result of this subsection, we need some more definitions.
	
	\begin{definition}
		A partition $\pi \in \ncp_n$ is called \emph{universal}, if $\pi$ has exactly one block $B$ with more than one element and the elements of $B$ are circularly consecutive in $P(\pi)$. 
		A chamber in $|\ncp_n|$ is called \emph{universal}, if every partition of the corresponding maximal chain in $\ncp_n$ is universal. Compare \cite[Def. 4.5]{hks}. An example for a universal chamber is shown in Figure \ref{fig:D-aptm}.
		
		For a non-crossing partition $\pi=\{B_1, \ldots, B_k\} \in \ncp_n$  and a subset $M \subseteq \{1, \ldots, n\}$, the partition \emph{induced} by $M$ is defined as the partition $\pi^M = \{B_1 \cap M, \ldots, B_k\cap M\}$. Note that the induced partition is again non-crossing and can be identified with a partition in $\ncp_{\#M}$. Induced chambers are defined analogously.
	\end{definition}
	
	\begin{prop}\label{prop:ncpn_union}
		Let $D$ in $|\ncp_n|$ be a chamber. Then $|\ncp_n|$ is the union of all apartments containing $D$ if and only if $D$ is a universal chamber.
	\end{prop}

	An example of a universal chamber $D$ in $|\ncp_4|$ and all apartments containing it is displayed in Figure \ref{fig:D-aptm}. The statement of the proposition can be verified by inspection of the order complex $|\ncp_4|$ shown in Figure \ref{fig:ncp4}.
	
	\begin{figure}[h]%
		\begin{center}
			\def\svgwidth{5cm}
			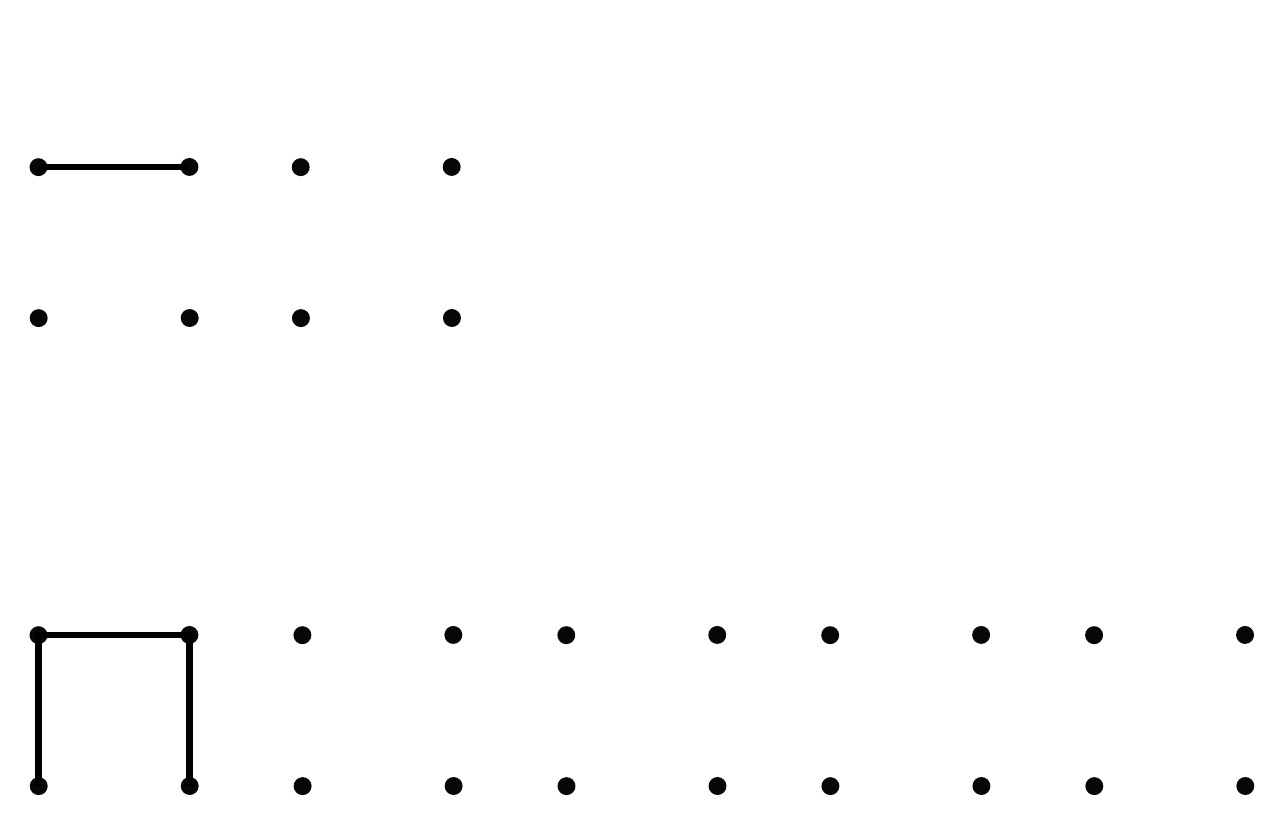
			\caption{A universal chamber $D$ in $|\ncp_4|$ and all apartments containing it.}%
			\label{fig:D-aptm}%
		\end{center}
	\end{figure}

	\begin{proof}
		The statement that $|\ncp_n|$ is the union of all apartments containing the chamber $D$ is equivalent to the following statement: For every chamber $C$ of $|\ncp_n|$ there is an apartment containing both $D$ and $C$. This means that there is a non-crossing spanning tree, such that both $D$ and $C$ arise as chambers of this apartment as described in Example \ref{ex:tree_chambers}.
		
		If $D$ is not a universal chamber, it is easy to construct a chamber $C$, such that they are not in a common apartment: there are crossings or cycles among the edges of the respective labeled spanning trees. 
		
		Now suppose that $D$ is a universal chamber and $C = (C_1, \ldots, C_{n-1})$ is an arbitrary chamber in $\ncp_n$ given in their pictorial representations, where $C_i$ is the $i$th partition in the chain defining $C$.  We argue by induction on $n$. The statement if trivial for $n=2$ and $3$, and for $n=4$, it is shown in Figure \ref{fig:D-aptm}. Suppose that $n > 4$. 
		
		Let $e=(i,j)$ be the unique edge in $C_1$.
		First, suppose that $i$ and $j$ are not consecutive on the circle. Then $e$ divides $\{1, \ldots, n\}$ into two parts $M=\{m:\ i \leq m \leq j\}$ and $N=(\{1, \ldots, n\} \setminus M) \cup \{i,j\}$.
		
		Let $C^M, C^N$ and $D^M, D^N$ be the chambers induced by $M$ and $N$, respectively. 
		Since $i$ and $j$ are not consecutive, we get that $\#M, \#N < n$ and we can use induction to find trees $T_M$ and $T_N$ corresponding to apartments containing the chambers $C^M$ and $D^M$, and $C^N$ and $D^N$, respectively. Since both trees contain the edge $e$, the merging of $T_M$ and $T_N$ along $e$ provides an apartment containing $C$ and $D$.
		
		If $i$ and $j$ are consecutive, consider the induced chambers on $M=\{1, \ldots, n\}\setminus\{j\}$ and use induction to get a tree $T$ for the induced chambers $C^M$ and $D^M$. The tree for the apartment containing $D$ and $C$ is then $T$ with the edge $e$ added.
	\end{proof}
	
	\begin{remark}
		The universal chambers in $\ocncw$ and the $M$-chains in $\ncw$ are in bijection via the canonical map $\tilde{g}$ between maximal chains in $\ncw$ and chambers in $\ocncw$ by Remark \ref{rem:bij_chains-chambers}.
	\end{remark}
	
	\begin{cor}\label{cor:Asupersolvable}
		The lattice $\ncp_n$ is supersolvable.
	\end{cor}
	
	\begin{proof}
		Let $D$ be a universal chamber in $|\ncp_n|$. Then every other chamber $C$ in $|\ncp_n|$ is in a common apartment $A$ with $D$. This means, that the corresponding maximal chains $c$ and $d$ are in a common Boolean lattice, which is distributive by definition. Hence, the sublattice generated by $c$ and $d$ is distributive and the assertion follows.
	\end{proof}
	
	\begin{remark}
		Björner and Edelman already showed  in \cite{bj} that $\ncp_n$ is shellable and for general types the analogous result for all finite Coxeter groups (i.e. that $\ncw$ is shellable) can be found in \cite{abw}.  An alternative way of proving shellability is to construct a linear ordering on the chambers, which is induced by the distance to a fixed universal chamber. 
	\end{remark}

	\subsection{Type $B$}\label{sec:typeB}
	The Coxeter group $W$ of type $B_n$ is the group of all \emph{signed permutations} of the set $\{1,\ldots,n,-1,\ldots,-n\}$. A permutation $\pi$ of $\{1, \ldots, -n\}$ is called \emph{signed}, if $\pi(-i)=-\pi(i)$ for all $i$. For further details see Chapter 8 of \cite{bb} or Section 3 of \cite{bra_watt_kpi}.
	
	We use the notation from \cite{bra_watt_kpi}. 
	Let $i_1, \ldots, i_k \in \{1, \ldots, -n\}$. If  $(i_1, \ldots, i_k)$ is disjoint from $(-i_1, \ldots, -i_k )$, then their product 
	is denoted by $\dka i_1, \ldots, i_k \dkz$.
	If we have $(i_1, \ldots, i_k,-i_1, \ldots, -i_k ) =(-i_1, \ldots, -i_k,i_1, \ldots, i_k )$, we will write $[i_1, \ldots, i_k]$ for this permutation.
	The set of reflections is 
	\[
	T(B_n)=\{\dka i,j\dkz :\ 1 \leq i < |j| \leq n\} \cup \{[i]:\ 1\leq i \leq n\}.
	\]

	Throughout this section, we will fix the Coxeter element $c=s_1\ldots s_{n-1} s_n=[1, \ldots, n]$ of $W$.
	
	As in type $A$, we start with explicitly constructing an embedding of  $|\nc(B_n)|$ into a finite spherical building. 
	A root system of type $B_n$ in $\R^n$ is given by
	\[
	\{\pm \varepsilon_i \pm \varepsilon_j \ | \ 1 \leq i < j \leq n\} \cup \{\pm\varepsilon_i \ | \ 1 \leq i \leq n \},
	\] 
	where $\varepsilon_i \in \R^n$ denotes the $i$th standard basis vector. Under the standard geometric representation, the reflection $[i]$ corresponds to the root $\varepsilon_i$, $\dka i, j \dkz$ corresponds to $\varepsilon_i - \varepsilon_j$ and $\dka i, -j \dkz$ corresponds to $\varepsilon_i + \varepsilon_j$ for positive $i$ and $j$. Let $e_1, \ldots, e_n$ be the standard basis vectors of $\F_3^n$.

	\begin{prop}\label{prop:bn_emb}
		The map $ f'\colon T(B_n) \to \L(\F_3^n)$ defined by 
		\[ t\mapsto 
		\begin{cases}
		\langle e_i\rangle &\text{if } t=[i], 1 \leq i \leq n\\
		\langle e_i - e_{j} \rangle &\text{if } t=\dka i,j \dkz,\ 1\leq i<j \leq n\\
		\langle e_i + e_{j} \rangle &\text{if } t=\dka i,-j \dkz,\ 1\leq i<j \leq n
		\end{cases}
		\]
		is injective and induces an embedding $f\colon |\nc(B_n)| \to |\L(F_3^n)|$ of the complex of non-crossing partitions of type $B_n$ into a finite building of type $A_{n-1}$ defined over $\F_3$.
	\end{prop}
	
	\begin{proof}
		The proof is the same as the proof for Proposition \ref{prop:An_emb}.
	\end{proof}
	
	We now recall the correspondence between the lattice $\nc(B_n)$ and the pictorial presentation from \cite{rei}. 
	
	\begin{definition}
		A partition $\pi = \{B_1, \ldots, B_k\}$ of $\{1,\ldots, n,-1,\ldots,-n\}$ is called a \emph{$B_n$-partition}, if the following two conditions are satisfied:
		\begin{itemize}
			\item[(i)]  $B \in \pi$ if and only if $-B \in \pi$
			\item[(ii)] there is at most one block $B \in \pi$ with $B=-B$ 
		\end{itemize}
		A block $B \in \pi$ with $B=-B$ is called \emph{zero block}.
	\end{definition}
	
	Since every $B_n$-partition can be identified with a partition of $\{1, \ldots, 2n\}$ via
	\[
	\{1,\ldots, n, -1, \ldots, n\} \to \{1, \ldots, 2n\},\quad i  \mapsto
	\begin{cases}
	i &\text{ if }i>0\\
	n-i &\text{ if }i<0
	\end{cases}
	\]
	the pictorial representation of a $B_n$-partition is defined analogously to the $A_n$-partition. Hence, for a $B_n$-partition $\pi$, its pictorial representation $P(\pi)$ is the convex hull of its blocks, where the vertices of the $2n$-gon are labeled with $1, \ldots, n, -1, \ldots, -n$ clockwise in this order. A $B_n$-partition $\pi$ is called \emph{non-crossing}, if there are no crossing
	blocks in $P(\pi)$. The set of non-crossing partitions of type $B_n$ is denoted by $\ncb_n$. This is a graded lattice, ordered by refinement and rank defined by $\rk(\pi)=n-\lfloor\frac{\#\pi}{2}\rfloor$ \cite[Prop. 2]{rei}. In analogy to type $A$, the join is the non-crossing span and the meet is given by intersection. Brady and Watt showed that the pictorial notion of non-crossing partitions coincides with the group theoretic definition of $\nc(B_n)$ \cite[Thm. 4.9]{bra_watt_kpi}.

	\begin{thm}\label{thm:bn_iso}
		For all $n \in \N$, the map
		\[
		\nc(B_n) \to \ncb_n, \quad w \mapsto P(\{w\})
		\]
		is a lattice isomorphism.
	\end{thm}
	
	\begin{example}
		The combination of Proposition \ref{prop:bn_emb} and Theorem \ref{thm:bn_iso} provides an embedding of $|\ncb_n|$ into the spherical building $\L(\F_3^n)$, which is defined via the images of edges in the $B_n$-partitions. 
		Figure \ref{fig:b3} illustrates how the complex $|\ncb_3|$ sits inside the building $|\L(\F_3^3)|$.
		Edges of the building that are not contained in $|\ncb_3|$ are shown in light gray. 
	\end{example}
	
	\begin{figure}%
		\begin{center}
			\includegraphics[width=9cm]{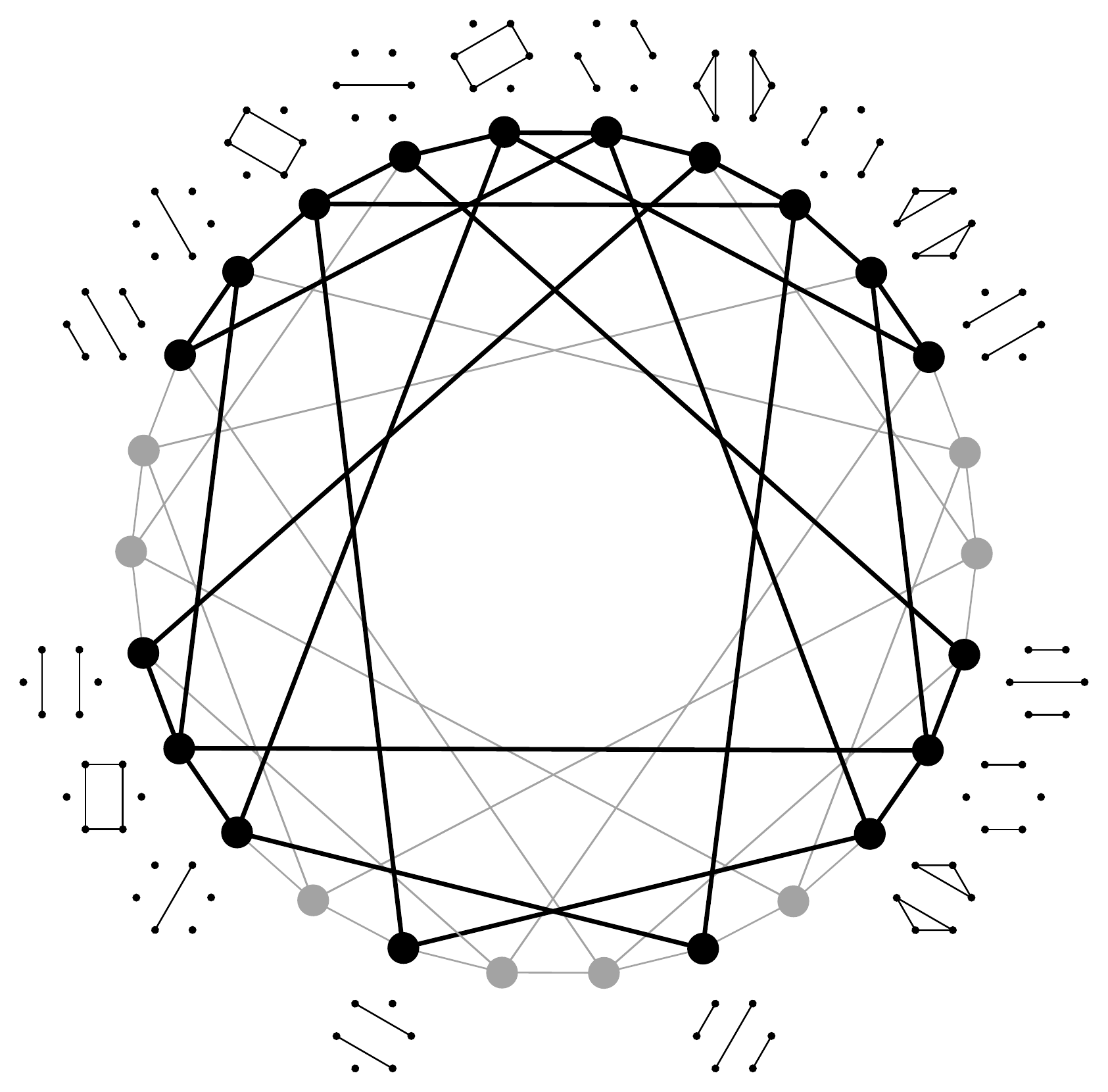}%
			\caption{The complex $|\ncb_3|$ embedded into the spherical building $|\L(\F_3^3)|$ of type $A_2$.}%
			\label{fig:b3}%
		\end{center}
	\end{figure}
	
	We now aim to describe the apartments of $|\ncb_n|$ as graphs on $2n$ vertices. We therefore introduce the notion of a $B_n$-graph.
	We denote the set of edges on the set $\{1, \ldots, n, -1, \ldots, -n\}$ by $\E_n^{\pm} \coloneqq \{(i,j):\ 1 \leq i < |j| \leq n\}$.
	
	\begin{definition}
		An edge of the form $(i, -i) \in \E_n^\pm$ is called \emph{zero edge}. A \emph{zero part} is either a collection of at most $n-1$ zero edges or the polygon spanned by a subset $B \subseteq \{1, \ldots, -n\}$ with $B=-B$. 
		
		Let $P$ be a regular $2n$-gon and $V$ its vertices, labeled $1, \ldots, n, -1, \ldots, -n$ clockwise in this order. Let $G$ be an embedded graph on the vertices $V$. 
		Then $G$ is called a \emph{$B_n$-graph}, if the following conditions are satisfied:
		\begin{enumerate}[label=(\alph*)]
			\item $G$ is invariant under $180^\circ$ rotation
			\item $G$ has exactly one zero part $Z$
			\item $G$ is a non-crossing forest outside $Z$
			\item the convex hull of $G$ equals $P$
		\end{enumerate}
	\end{definition}
	
	\begin{prop}
		The apartments in $\ncb_n$ are in one-to-one correspondence with $B_n$-graphs.
	\end{prop}
	
	The proof is the exact analog of the proof of Proposition 4.4 of \cite{hks}. 
	
	\begin{remark}
		The complex $|\ncb_n|$ is \emph{not} a building if $n>2$. As in type $A$, there are simplices in $|\ncb_n|$, which are not contained in a common apartment. Figure \ref{fig:b4-elements} shows examples for this fact.
	\end{remark}
	
	\begin{example}
		Figure \ref{fig:B3_aptm} shows all apartments containing the chamber corresponding to $\dka 1,2 \dkz \leq [1,2]$. Note that there is the following correspondence between reflections in $T(B_n)$ on the left-hand side and edges or pairs of edges in $\E_n^\pm$ on the right-hand side:
		\begin{align*}
		\dka i,j \dkz  &\longleftrightarrow (i, j),(-i,-j)\\
		[i] &\longleftrightarrow (i, -i)
		\end{align*}
	\end{example}
	
	\begin{figure}%
		\begin{center}
			\def\svgwidth{8cm}
			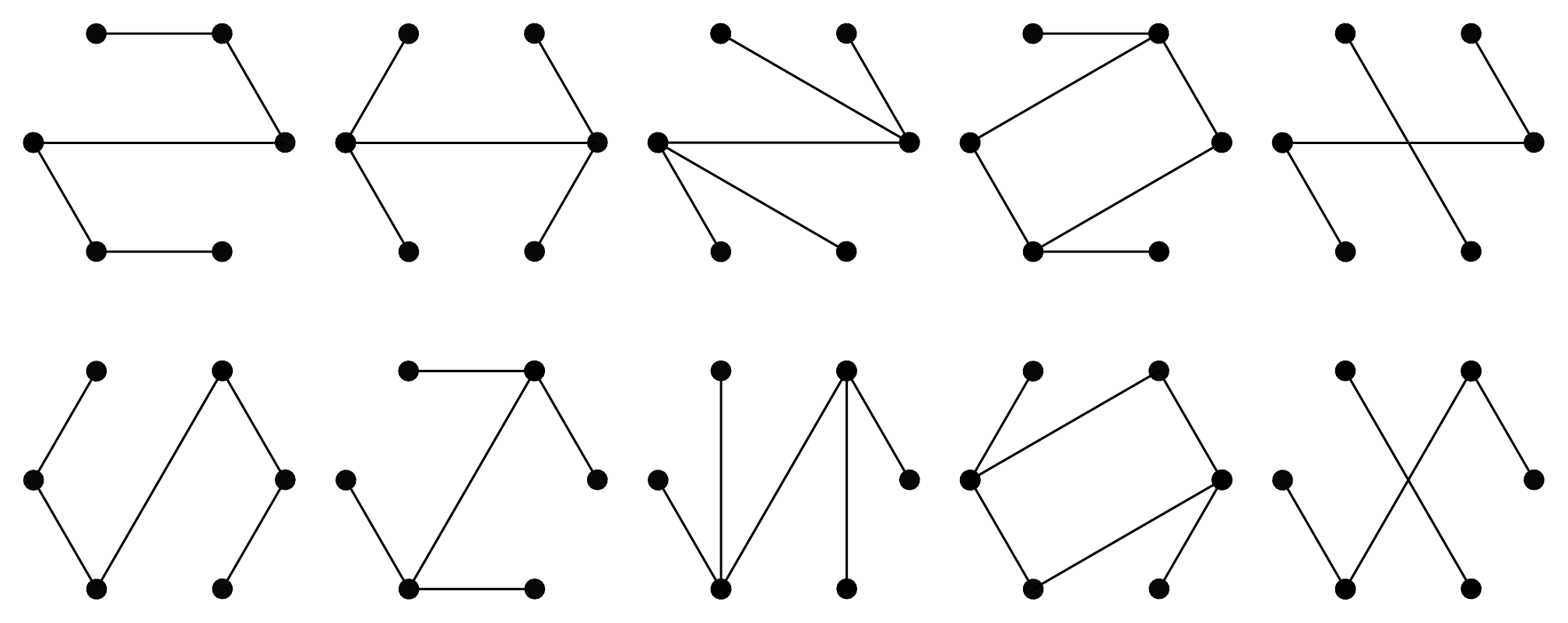
			\caption{All apartments of $\ncb_3$ containing the chamber corresponding to $\dka 1,2 \dkz \leq [1,2]$.}%
			\label{fig:B3_aptm}%
		\end{center}
	\end{figure}
	
	In order to investigate the supersolvability in type $B$, we need a characterization of supersolvability in terms of left modularity. Let $L$ be a lattice. An element $x \in L$ is \emph{left modular}, if  for all $y \leq z$
	\[
	(y \vee x) \wedge z = y \vee (x \wedge z).
	\]
	The lattice $L$ is \emph{left modular} if there is a maximal chain consisting of left modular elements. The following characterization is due to \cite[Thm. 2]{mcn_tho}.
	
	\begin{thm}
		A finite graded lattice supersolvable if and only if it is left modular. In particular, every $M$-chain consists of left modular elements.
	\end{thm}
	
	Now we are in the position to prove the following.
	
	\begin{thm}\label{thm:ncbn_ss}
		The lattice $\ncb_n$ is only supersolvable if $n \leq 3$.
	\end{thm}
	
	\begin{proof}
		The complex $|\ncb_2|$ is a union of four points, hence $\ncb_2$ is supersolvable as apartments correspond to pairs of points. If $n=3$, Figure \ref{fig:b3} displays $|\ncb_3|$. Inspection of this figure shows that the union of all apartments containing the chamber corresponding to $\dka 1,2\dkz \leq [1,2]$ (compare Figure \ref{fig:B3_aptm}) is all of $|\ncb_3|$. Hence,  $\dka 1,2\dkz \leq [1,2]$ is an $M$-chain in $\ncb_3$, which shows supersolvability.
		
		Now let $n \geq 4$. Suppose there is an $M$-chain in $\ncb_n$ corresponding to a chamber $C$. Then every $B_n$-partition of $C$ has to be in a common apartment with every other element in $|\ncb_n|$. We show that there is no $M$-chain in $\ncb_n$ by showing that there is no left modular element of rank $2$ in $\ncb_n$.
		
		Figure \ref{fig:b4-elements} displays all classes of rank 2 elements of $\ncb_4$ in black and respective obstructions to left modularity in grey. These partitions can be easily generalized for $n \geq 5$. This proves the claim.
	\end{proof}
	
	\begin{figure}[h]%
		\begin{center}
			\includegraphics[width=12cm]{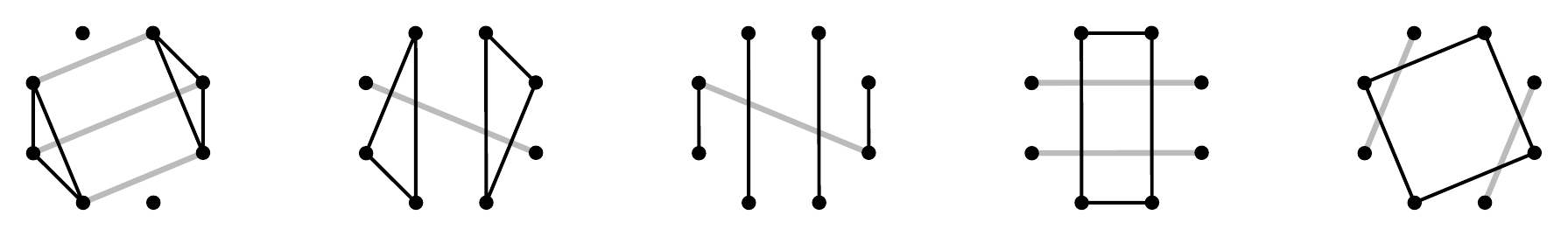}
			\caption{Obstructions to left modularity in $\ncb_4$.}%
			\label{fig:b4-elements}%
		\end{center}
	\end{figure}
	
	As a direct consequence, we get the following.
	
	\begin{cor}
		For $n\geq 4$ there is no chamber $B\in\chambers(|\ncb_n|)$ such that the union of all apartments containing $B$ is $|\ncb_n|$.
	\end{cor}

	Without going into the details, we want to mention that the non-crossing partitions of type $D$ can also be interpreted in a pictorial way. This was introduced in \cite{ath_rei}. With similar considerations, it follows that $\nc(D_n)$ is only supersolvable for $n \leq 3$ as $\nc(D_3) \cong \nc(A_3)$. 
	

	\section{Hurwitz- and Chamber Graphs}\label{sec:Hurwitz}
	
	In this section we show that the Hurwitz graph $H(W)$ of a finite Coxeter group is a subgraph of the chamber graph $\Gamma_\Delta$ of a spherical building $\Delta$, which is  a consequence of the Main Theorem. We use this to compute in Theorem \ref{thm:radius} bounds on the radius and diameter of $\hw$ for all finite $W$, which partially answers Question 11.3 of \cite{ar}. Moreover, we give pictorial interpretations of the shifts $\sigma_i$ in type $A$.

	The \emph{chamber graph} $\Gamma_\Delta$ of a chamber complex $\Delta$ is the simplicial graph with vertices the chambers of $\Delta$ and two different vertices are connected by an edge, if they have a common codimension one face. The chamber graph is connected by definition of chamber complexes. Moreover, note that if $X$ is a chamber subcomplex of $\Delta$, then $\Gamma_X$ is a subgraph of $\Gamma_\Delta$.
	
	\begin{definition}\label{def:hurwitz}
		Let $W$ be a finite Coxeter group. 
		The \emph{Hurwitz graph} $H(W)$ is the undirected graph whose vertices are the maximal chains $\chains(W)$, where two chains are connected by an edge if they differ in exactly one element.
	\end{definition}

	\begin{remark}
		Note that \cite{ar} contains more than one definition of a Hurwitz graph. Our Definition \ref{def:hurwitz} differs from Definition 11.2 in \cite{ar} for arbitrary finite Coxeter groups $W$.  In type $A$ the two definitions coincide, but in general they do not. The set of vertices are in bijection, but we allow more edges. Our definition seems to be more natural and should be seen as a generalization of the definition of the type $A$ Hurwitz graph, which is Definition 2.2 in \cite{ar}. 
	\end{remark}
	
	In the following, we will work with Definition~\ref{def:hurwitz}. 
	
	\begin{prop}\label{prop:isom_graphs}
		Let $W$ be a finite Coxeter group of rank $n$. The Hurwitz graph $H(W)$ is isomorphic to a connected subgraph of the chamber graph $\Gamma_\Delta$ of a spherical building $\Delta$ of type $A_{n-1}$. 
	\end{prop}
	
	\begin{proof}
		By Remark~\ref{rem:bij_chains-chambers}, the maximal chains (i.e. the vertices of $\hw$) are in bijection with the maximal chambers of $\ocncw$, which are the vertices of its chamber graph $\Gamma_{\ocncw}$. By Corollaries \ref{cor:adj_chains} and \ref{cor:adj_chambers}, two vertices of $\hw$ are connected by an edge if and only their images under the bijection are. Hence, $\hw \cong \Gamma_{\ocncw}$.  As $\ocncw$ embeds into a spherical building $\Delta$ of type $A_{n-1}$ by our Main Theorem, the assertion follows.
	\end{proof}

	\subsection{Diameter and radius of $H(W)$}
	In this subsection we give an estimate on the radius of $H(W)$ for arbitrary finite $W$ and compute it explicitly in a couple of cases. Recall that the \emph{eccentricity} of a vertex $v$ of a graph $G=(V,E)$ with vertices $V$ and edges $E$ is
	\[
	\ecc(v):=\max_{w\in V} d(v,w),
	\]
	where $d$ denotes the standard metric on the vertex set. The \emph{diameter} of $G$ is the maximal eccentricity
	\[
	\diam(G) \coloneqq \max_{v\in V}\ecc(v)
	\]
	and the \emph{radius} of $G$ is the minimal eccentricity
	\[
	\rad(G) \coloneqq \min_{v\in V}\ecc(v).
	\]

	One obviously always has the inequality $\rad(G) \leq \diam(G)$. Our main result is the following theorem. 
	
	\begin{thm}\label{thm:radius}
		Let $W$ be a finite Coxeter group of rank $n$. Then the radius of the Hurwitz graph of $W$ satisfies 
		\[ 
		\binom{n}{2}\leq \rad( H(W)). 
		\]
		If $W$ is of type $A_n$, then $\rad( H(W))= \binom{n}{2}$. Moreover, we have for $W$ of type $A_3$ that $\diam(H(W))=3$ and for $W$ being of type $A_4$ that ${\diam(H(W))=7}$.
		If $W$ is of type $B_3$, then $\rad(H(W))= \binom{3}{2}=3$ and $\diam(H(W))= 4$.
	\end{thm}
	
	For type $A$ the statement of the theorem is \cite[Thm. 10.3]{ar}. 
	For the proof of our theorem we need the following lemma. 
	
	\begin{lemma}\label{lem:rad-apartment}
		For any vertex $v$ in the chamber graph $\Gamma=\Gamma_\Sigma$ of a Coxeter complex $\Sigma$ of type $A_k$ we have 
		\[
		\ecc(v)=\diam(\Gamma)=\rad(\Gamma)=\binom{k+1}{2}. 
		\] 
		The same is true for spherical buildings $\Delta$ of type $A_k$.
	\end{lemma}
	
	\begin{proof}
		First consider the case of a Coxeter complex. From standard Coxeter combinatorics it follows that the eccentricity of any vertex $v$ equals the length of the longest word in the corresponding Coxeter group $W \cong S_{k+1}$. Hence, $\ecc(v)=\binom{k+1}{2}$ for all $v$, which implies $\ecc(v)=\diam(\Gamma)=\rad(\Gamma)$ for all $v$. 
		To see that the statement on $\Delta$ is true, recall that any pair of chambers $C, D$ in $\Delta$ is contained in a common apartment isomorphic to $\Sigma$. 
	\end{proof}
	
	The following corollary is immediate. 
	
	\begin{cor}
		Denote by $\Gamma_A$ the subgraph of the chamber graph $\Gamma_\Delta$ that corresponds to an apartment $A$ in $\Delta$. Then 
		the ball of radius $\binom{n}{2}$ around a vertex $v$ in $H(W)$ contains all $\Gamma_A$, where $A$ is an apartment containing the chamber $C_v$ in $\Delta$ corresponding to $v$. 
	\end{cor}
	
	Observe that the other direction is not true (not even in type $A$). In general, the ball of radius $\binom{n}{2}$ around a random vertex $v$ in $H(W)$ is {\bf not} contained in the union of all $\Gamma_A$, where $A$ is an apartment containing the chamber $C_v$ in $\Delta$ corresponding to $v$. In type $A$, it is however true for vertices corresponding to universal chambers by Proposition \ref{prop:ncpn_union}.

	\begin{proof}[Proof of Theorem~\ref{thm:radius}]
		By Proposition \ref{prop:isom_graphs} we can identify $\hw$ with the graph $\Gamma_{\ocncw}$. From Lemma ~\ref{lem:rad-apartment} (with $k=n-1$) we obtain that for every apartment $A$ in $\ocncw$ the radius of the subgraph $\Gamma_A$ is $\binom{n}{2}$. This implies that the eccentricity of a vertex $v$ in $\Gamma_A$ is $\binom{n}{2}$. Now $\ocncw$ is a union of apartments by Corollary \ref{cor:ncw_union_aptm}, so every vertex of $\Gamma_{\ocncw}$ is contained in a subgraph $\Gamma_A$ for some apartment $A$. Therefore, $\rad(\hw) \geq \binom{n}{2}$, as the radius is the minimal eccentricity.
		
		In type $A$ we know by Proposition \ref{prop:ncpn_union} 
		that there exists a universal chamber $B$, which has the property that every other chamber is contained in a common apartment with $B$. Therefore, we have equality for the radius in type $A$. With the same arguments and Theorem \ref{thm:ncbn_ss} equality follows for $B_3$. 
		
		The equalities for the diameter are by inspection. For $|\ncp_4|$ it can be verified in Figure \ref{fig:ncp4} that the eccentricity is $3$ for every chamber. In $|\ncp_5|$ the maximal eccentricity is for instance realized by the chambers corresponding to $(13)(45)(12)(35)$ and $(24)(15)(23)(14)$. In Figure \ref{fig:b3} one tests that the maximal eccentricity is for instance realized by the chambers corresponding to $\dka 1, -2 \dkz \dka 2,3 \dkz [1]$ and $[1]\dka 2,3\dkz \dka 1,-2 \dkz$.
	\end{proof}

	\subsection{The Hurwitz action in type A}
	
	In  this subsection we give a pictorial interpretation of the Hurwitz action and interpret  left- and right- shifts $\sigma$ and $\sigma'$ in terms of ``sliding edges'' in labeled spanning trees representing chambers. A similar pictorial interpretation can be given for the Hurwitz action in type $B_n$ using $B_n$-trees.
	
	Since the Hurwitz graph $\hw$ is isomorphic to the chamber graph of $\ocncw$ by Proposition \ref{prop:isom_graphs}, the vertices of $\hw$ correspond to chambers in $\ocncw$. From now on, we will use this correspondence frequently.
	
	Let $W\cong S_n$ and recall from Section~\ref{Sec:typeA} that the reflections in $W$ are the transpositions $(i,j)$ with $i\neq j\in\{1, \ldots, n\}$.  These transpositions can be represented by an edge connecting the vertices numbered $i$ and $j$ in a regular $n$-gon. 
	
	Lemma~\ref{lem:ch_aptm} implies that every chamber $C$ in $\ocncw$ comes with a preferred apartment $A_C$ which is given by a non-crossing spanning tree $T(C)$. We use the fact that $C$ is given by a minimal presentation $t_1\ldots t_n$ of the Coxeter element $c$ and define $T(C)$ to be the tree containing all edges $(j_i, k_i)$  corresponding to the transpositions $t_i$.  We may label the edge $(j_i, k_i)$ by  $i$.
	This labeling then determines $C$ uniquely. 
	
	\begin{lemma}\label{lem:no-common-vertex}
		Suppose $T$ is a non-crossing labeled spanning tree representing a chamber $C$ in $H(S_n)$. Two edges in $T$ do not intersect if and only if the corresponding transpositions commute.   
	\end{lemma}
	\begin{proof}
		If the edges do not intersect in $T$, then they represent two transposition $(i,j)$ and $(k,l)$ with $i,j,k$ and $l$ all pairwise distinct. Hence they commute. Similarly obtain the converse. 
	\end{proof}
	
	The next proposition gives a characterization of the braid group action on the Hurwitz graph in terms of trees and could be taken as an alternative definition in type $A$. The second author discussed this idea first with Thomas Haettel and Dawid Kielak in 2014. 
	
	\begin{prop}\label{prop:slides}
		Suppose $T$ is a non-crossing labeled  spanning tree representing the chamber $C$ in $H(S_n)$ and let $t_i=\{j,k\}$ and $t_{i+1}=\{l,m\}$ be the edges labeled $i$ and $i+1$, respectively.  \newline Then a non-crossing labeled spanning tree $\sigma_i(T)$ representing $\sigma_i(C)$ can be obtained from $T$ as follows:
		\begin{itemize}
			\item If the edges labeled $i$ and $i+1$ do not share a vertex (i.e. j, k, l, m are pairwise different), swap their labeling and keep all other edges.  
			\item If the edges  labeled $i$ and $i+1$ do share a vertex, say $k=m$, then label $\{j,k\}$ with $i+1$ and replace the edge $\{l,m\}$ by the edge $\{l, j\}$ and label it with $i$.  Keep all other edges. 
		\end{itemize}
		A non-crossing labeled spanning tree $\sigma'_i(T)$ representing $\sigma'_i(C)$ can be obtained from $T$ as follows:
		\begin{itemize}
			\item If the edges labeled $i$ and $i+1$ do not share a vertex (i.e. j, k, l, m are pairwise different), swap their labeling and keep all other edges.  
			\item If the edges  labeled $i$ and $i+1$ do share a vertex, say $k=m$, then label $\{l,m\}$ with $i$ and replace the edge $\{j,k\}$ by the edge $\{j,l\}$ and label it with $i+1$.  Keep all other edges. 
		\end{itemize}
	\end{prop}

	\begin{proof}
		We only prove the assertion for $\sigma_i$. The rest follows using analogous arguments. 
		Lemma~\ref{lem:no-common-vertex} implies the first bullet point. To see the second, argue as follows. 
		From Lemma~\ref{lem:shift} we know that if two minimal presentations differ by a single shift, then the corresponding sequence of  transpositions differ in the $i$-th and $i+1$-st entry. 
		Moreover, we know that the $i+1$-st transposition in $\sigma_i(C)$ equals the $i$-th transposition in $C$, as $\sigma_i$ shifts the $i$-th transposition to the right. In terms of edges in the  tree $T$ this means that the edge $\{l,m\}$ stays and obtains a new label $i+1$ in $\sigma_i(T)$. 
		The transposition at index $i$ in $C$ is replaced by the $t_i$-conjugate of $t_{i+1}$. Recall that $t_i$ swaps $j$ and $k$ while $t_{i+1}$ swaps $l$ and $k=m$. Hence $t_it_{i+1}t_i$ swaps $j$ and $l$. Hence the tree $\sigma_i(T)$ contains an edge $\{j,l\}$ labeled $i$. 
		Finally, to see that the tree $\sigma_i(T)$ is non-crossing observe that the convex hull in the regular polygon on the three vertices $j, k=m, l$ is a triangle that is not cut by any other edge of $T$. Compare Figure~\ref{fig:spanning3}. Hence, we may choose any two of the bounding edges and still obtain a non-crossing tree.   
	\end{proof}
	
	\begin{figure}[h]%
		\begin{center}
			\def\svgwidth{9cm}
			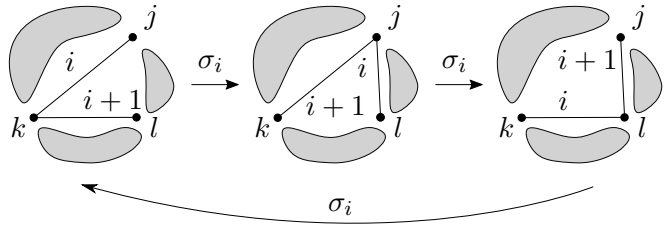
			\caption{This illustrates the orbit of  $\sigma_i$ in case the $i$-th and $i+1$-st transposition of a maximal chain do not commute. The gray shaded regions represent the remaining vertices which are connected to $j$, $k$ and $l$ in such a way that none of the edges crosses the triangle spanned by $j$, $k$ and $l$. }
			\label{fig:spanning3}%
		\end{center}
	\end{figure}
	
	Let us illustrate the assertion of Proposition~\ref{prop:slides} with an example. 
	
	\begin{example}\label{ex:slides}
		Suppose we would like to apply $\sigma_i$ to a given non-crossing labeled spanning tree in case $n=5$.  Figure~\ref{fig:sigmas} shows such an example. In the pictures we omit all the labels of the vertices as well as the labels of the edges which do not affect the application of $\sigma_i$.    
		
		The situation shown is the one where the edges labeled $i$ and $i+1$ do intersect in a vertex $k$. In this case the new tree $\sigma_i(T)$ is obtained  from the old tree by first swapping the  labels $i$ and $i+1$ and then ``sliding'' the edge labeled $i$ along the edge labeled $i+1$ so that they now share the other vertex of the edge (newly named) $i+1$. Such a slide move corresponds to the conjugation of the edge labeled $i+1$ by the edge labeled $i$. 
		
		The image $\sigma'_i$ is the inverse of the map $\sigma_i$. Here we first swap labels of the edges and then slide edge number $i+1$ along edge $i$, which ``undoes'' what $\sigma_i $ changed in the labeled spanning tree. 
	\end{example}
	
	\begin{figure}%
		\begin{center}
			\def\svgwidth{\textwidth}
			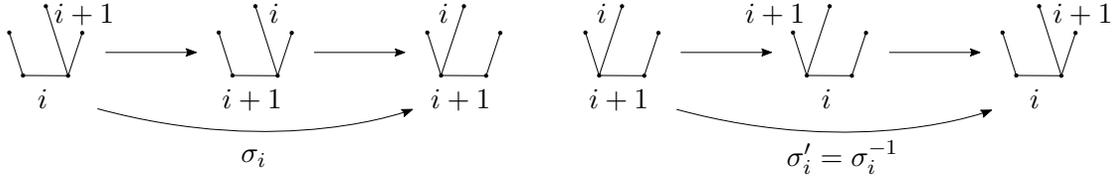
			\caption{This figure illustrates how trees characterizing vertices in $H(S_5)$ change under an application of $\sigma_i$ and $\sigma'_i$. For details see Example~\ref{ex:slides}.}%
			\label{fig:sigmas}%
		\end{center}
	\end{figure}
	
	Using the pictorial interpretation of the $\sigma_i$ given in Proposition~\ref{prop:slides} and  a characterization of all the non-crossing labeled spanning trees that correspond to a vertex in $H(S_n)$ given in Theorem 2.2 of \cite{gy} (see also \cite[Prop. 3.5.]{ar}), one can give  (yet another) proof of the fact that the Hurwitz graph is connected. 
	As it very much depends on the pictorial interpretation of type $A$, we only sketch how to proceed. 
	
	\begin{thm}
		The Hurwitz graph $H(S_n)$ is connected. 
	\end{thm}
	
	\begin{proof}[Sketch of proof.]
		Let $\{1,\ldots, n\}$ be the vertices of a regular polygon enumerated clockwise. The following non-crossing labeled spanning tree $T_0$ represents a universal chamber $v_0$ in the set of vertices of the Hurwitz graph: Let $T_0$ be the tree containing all outer edges $\{i,i+1\}$ of the polygon for $i=1, \ldots, n-1$ where the edge $\{i,i+1\}$ is labeled with $i$.
		
		We will argue that any tree $T$ representing a vertex of $H(S_n)$ can be deformed into $T_0$ by a sequence of $s_i$'s. This then gives us edge-paths in $H(S_n)$ connecting every vertex $v$ with $v_0$. Following such a path backwards using  $\sigma'_i$ applications in place of $\sigma_i$, we obtain connectivity. 
		
		We will write $e_i$ for the edge labeled $i$ in any given tree. First observe that one can deform any tree into a tree containing only outer edges by applying $\sigma_i$'s. The reason is the following: any such labeled spanning tree contains at least one outer edge $e_i$  and suppose that $e_{i+1}$ shares a vertex with $e_i$ and is not an outer edge. Then we can use (repeated) applications of $\sigma_i$'s and $\sigma'_i$'s to deform the tree such that we move the edge labeled $i+1$ previously to an outer edge. Compare Figure~\ref{fig:sigmas}.

		In case $e_{i+1}$ does not share a vertex with $e_i$, we can swap labels of non-intersecting  edges in order to obtain a situation for which we have intersecting edges that share consecutive labels. Note that this step requires $n\geq 5$. This is not a big restriction as the case $n=4$ can be checked by hand.

		Suppose we have moved all edges so that the spanning tree only contains outer edges. It remains to see that we can move the edges around the circle so that the missing edge is $\{n-1, 1\}$. 
		Compare Figure~\ref{fig:gap_change} for how to do this. 
		
		Finally, we can rearrange the labels on the edges so that we arrive in the tree $T_0$ representing the universal chamber $v_0$. Again as $n>5$, we have enough room to swap even labels of adjacent edges. This step is shown in Figure~\ref{fig:label_change}. 
	\end{proof}
	
	\begin{figure}[h]
		\begin{center}
			\def\svgwidth{9cm}
			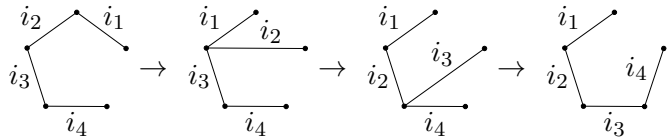
			\caption{Rearranging the outer edges in a labeled spanning tree using slide moves.}%
			\label{fig:gap_change}%
		\end{center}
	\end{figure}
	
	\begin{figure}
		\begin{center}
			\def\svgwidth{9cm}
			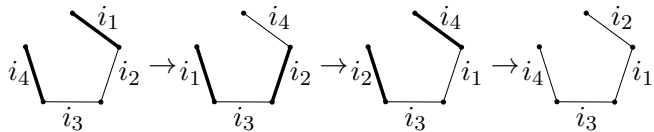
			\caption{Swapping the labels of the adjacent edges $i_1$ and $i_2$.}%
			\label{fig:label_change}%
		\end{center}
	\end{figure}

	The property of the shifts $\sigma_i$ stated in the following corollary was already observed in \cite{ar} and can easily be obtained directly from the definition of $\sigma_i$. We include it here as one can also prove it using the labeled-tree slide-move characterization of the $\sigma_i$ as well.  
	\begin{cor}\label{cor:order}
		If the edges labeled $i$ and $i+1$ in a tree $T$ do not share a vertex, $\sigma_i^2(T)=T$ and if they do share a vertex, then $\sigma_i^3(T)=T$.
	\end{cor}
	\begin{proof}
		We use notation as in the proof of Proposition \ref{prop:slides}.  If they do share a vertex of $T$, say vertex $k$, then $\sigma_i$ is transitive on all spanning subtrees on the three vertices $j,k,l$ of these two edges. As no other edge of $T$ crosses the convex hull of $j,k,l$ the resulting trees are all non-crossing. Compare Figure~\ref{fig:spanning3}.  
	\end{proof}
	
	\appendix
	
	\section{Spherical Buildings of type $A_n$}\label{cha:buildings}
	
	In this appendix we collect some basic facts about spherical buildings. We do not go much into the details but refer the reader to Brown's book \cite{bro} or to the book of Abramenko and Brown \cite{ab} instead. 
	
	\begin{definition}
		A simplicial complex $\Delta$ is a \emph{building} if there exists a collection $\A(\Delta)$ of subcomplexes of $\Delta$, called \emph{apartments}, such that the following axioms are satisfied. 
		\begin{itemize}
			\item[(B0)]
			Each apartment is a Coxeter complex.
			\item[(B1)]\label{ax:b2}
			For any two simplices in $\Delta$ there is an apartment containing both.
			\item[(B2)]
			If two simplices $a, b \in \Delta$ are both contained in the apartments $A$ and $A'$, then there is an isomorphism $A \to A'$ fixing $a$ and $b$ pointwise.
		\end{itemize}
		The maximal simplices in $\Delta$ are called \emph{chambers}. 
	\end{definition}
	
	It follows from (B2) that any two apartments are isomorphic. This allows us to define the \emph{type} of a building, which is the type of the Coxeter group associated to (one of) its apartments. A building is called \emph{spherical}, if the apartments are spherical Coxeter complexes, i.e. arise from a finite Coxeter group. An example of a building is given in Figure~\ref{fig:F2-building}.%
	
	Coxeter Complexes are \emph{chamber complexes}, i.e. all maximal simplices have the same dimension and any two of them can be connected by \emph{gallery}. A gallery in $\Delta$ is a sequence of chambers $(C_1, \ldots, C_n)$ such that $C_i$ and $C_{i+1}$ are either equal or share a codimension one face. In particular, chamber complexes are connected. Axioms (B2) and (B3) imply that buildings are chamber complexes as well. The set of chambers of a building $\Delta$ is denoted by $\chambers(\Delta)$.
	
	A subcomplex $Y$ of a chamber complex $X$ is called \emph{chamber subcomplex}, if it is a chamber complex and $\dim(X)=\dim(Y)$. Hence, apartments are chamber subcomplexes of the building.
	
	\begin{figure}[h]%
		\begin{center}
			\def\svgwidth{9.4cm}
			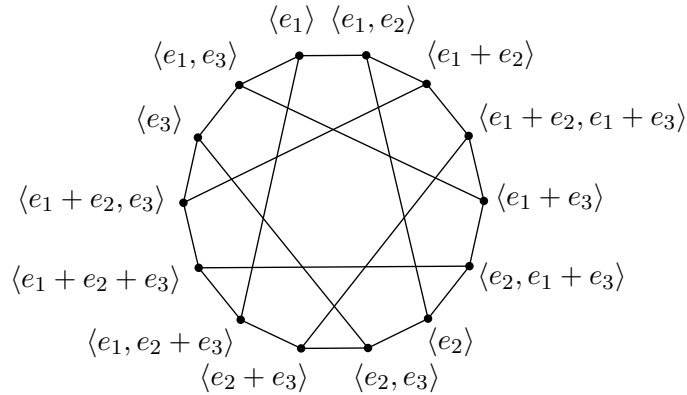
			\caption{The spherical building $|\L(\F_2^3)|$ of type $A_2$. The vertices are labeled by the linear subspaces they represent. The vector $e_i$ is the $i$-th unit vector. }%
			\label{fig:F2-building}%
		\end{center}
	\end{figure}
	
	The next corollary is an immediate consequence of the axiom (B1).
	\begin{cor}\label{cor:building_union}
		Let $C$ be an arbitrary chamber of a building $\Delta$. Then $\Delta$ is the union of all apartments containing $C$.
	\end{cor}
	
	The following construction is a fundamental example of a spherical building of type $A_{n-1}$. For more details and a proof of Proposition~\ref{prop:LV_building} see \cite[chap. 4.3]{ab}.
	
	Let $V$ be a vector space of dimension $n\geq 2$ and $\L(V)$ the lattice of linear subspaces of $V$. Recall that $\L(V)$ is ordered by inclusion and ranked by dimension. The meet is given by intersection and the join by the sum of subspaces. The minimal element is $\{0\}\subseteq V$ and the maximal element is $V$.
	
	For a lattice $L$, the order complex $|L|$ is the abstract  simplical complex with vertices $L\!\setminus\!\{0,1\}$ and the simplices are the chains in $L$. 
	
	\begin{prop}\label{prop:LV_building}
		For any vector space $V$ of dimension $n\geq 2$  the order complex  $|\L(V)|$ of the lattice of linear subspaces of $V$ is a spherical building of type $A_{n-1}$.
	\end{prop}
	
	We now want to describe the apartments and chambers in terms of $V$ and start with the Coxeter complex of type $A_{n-1}$.
	
	The Coxeter complex $\Sigma$ of type $A_{n-1}$ is a simplicial complex isomorphic to the barycentric subdivision of the boundary of an $(n-1)$-simplex $\Delta_{n-1}$ denoted by $\sd(\partial \Delta_{n-1})$. The vertices of $\sd(\partial \Delta_{n-1})$ can be labeled with proper nonempty subsets of $\{1, \ldots, n\}$ where the cardinality of a vertex label is the rank of the corresponding face in $\partial \Delta_n$. The faces of $\sd(\partial \Delta_{n-1})$ are labeled with chains of labels of its vertices. These are other words to say that $\Sigma$ is isomorphic to a Boolean lattice $\B_n$, i.e. the order complex of the power set of $\{1, \ldots, n\}$ ordered by inclusion.

	The apartments of $\L(V)$ correspond to sublattices of $\L(V)$ that are isomorphic to Boolean lattices generated by $n$ elements. These $n$ elements are one-dimensional subspaces $L_1, \ldots, L_n$ such that $L_1 \oplus \ldots \oplus L_n = V$. Such a set $\{L_1, \ldots, L_n\}$ is called \emph{frame} of $V$.
	We summarize these properties in the following proposition.
	
	\begin{prop}
		Let $V$ be an n-dimensional vector space.
		Then the following is true. 
		The frames of $V$  are in one-to-one correspondence with apartments of $|\L(V)|$. Moreover, every basis of $V$ determines a frame and hence an apartment of $\L(V)$. 
		Chambers in the building $|\L(V)|$ correspond to maximal chains in $\L(V)$ 
		and the chambers of a fixed apartment $\Sigma$ are given by total orderings of the frame associated with $\Sigma$.
	\end{prop}

	\subsection*{Acknowledgements}
	We would like to thank Christophe Hohlweg for helpful discussions.  
	The first author would like to thank Russ Woodroofe to pointing out to her the concept of supersolvability and helpful discussions in Herstmonceux. The second author would like to thank Thomas Haettel and Dawid Kielak for the many hours spent  in the John Todd Wing at the  Oberwolfach Research Institute for Mathematics (MFO Research In Pairs ID 1421q)  discussing braid groups with her.   
	In addition we thank the anonymous referee for a thoughtful report and for pointing  out the reference \cite{her}.

\bibliography{bibliography}
\bibliographystyle{alpha}

\end{document}